\documentclass[11pt, oneside]{article}   	
\usepackage{geometry}                		
\usepackage[parfill]{parskip}    		
\usepackage{graphicx}				

\graphicspath{ {images/} }

\newcommand{\bigO}{\mathcal{O}}
\newcommand{\half}{\frac{1}{2}}
\newcommand{\R}{\mathbb{R}}

\newcommand{\N}{\mathbb{N}}

\renewcommand{\vec}[1]{\bm{#1}}
\newcommand{\unitvec}[1]{\hat{\bm{#1}}}
\newcommand{\ddx}[1]{{{\rm d} \over {\rm d} #1}}
\newcommand{\ppx}[1]{{\partial \over \partial #1}}
\newcommand{\pfpx}[2]{{\partial #2 \over \partial #1}}
\newcommand{\pmpxm}[2]{{\partial^{#2} \over \partial {#1}^{#2}}}

\newcommand{\costheta}{\cos\theta}
\newcommand{\sintheta}{\sin\theta}
\newcommand{\cosphi}{\cos\varphi}
\newcommand{\sinphi}{\sin\varphi}

\newcommand{\grad}{\nabla}
\newcommand{\DeltaS}{\Delta_{\rm S}}

\newcommand{\calB}{\mathcal{B}}

\newcommand{\calL}{\mathcal{L}}

\newcommand{\Dnt}{D^\top_n}

\newcommand{\genjac}{R}
\newcommand{\genjacnmk}{\genjac_{n-k}}
\newcommand{\genjacmmj}{\genjac_{m-j}}
\newcommand{\genjacw}{w_\genjac}

\newcommand{\normgenjac}{\omega_\genjac}

\newcommand{\bigW}{\mathbb{W}}

\newcommand{\scop}{Q}
\newcommand{\scopnki}{\scop_{n,k,i}}
\newcommand{\scopmjh}{\scop_{m,j,h}}
\newcommand{\scopa}{\scop^{(a)}}
\newcommand{\scopnkia}{\scopnki^{(a)}}
\newcommand{\scopmjha}{\scopmjh^{(a)}}
\newcommand{\bigscop}{{\mathbb{\tilde Q}}}
\newcommand{\bigscopa}{\bigscop^{(a)}}
\newcommand{\bigscopN}{\bigscop_{N}}
\newcommand{\bigscopNa}{\bigscopa_{N}}
\newcommand{\bigscopNka}{\bigscopa_{N,k}}
\newcommand{\bigscopt}{\mathbb{{Q}}}
\newcommand{\bigscopta}{\bigscopt^{(a)}}

\newcommand{\bigscoptna}{\bigscopta_{n}}
\newcommand{\xvec}{\mathbf{x}}
\newcommand{\ch}{Y}
\newcommand{\chki}{\ch_{k,i}}
\newcommand{\chjh}{\ch_{j,h}}

\newcommand{\alphaa}{\alpha^{(a)}}
\newcommand{\betaa}{\beta^{(a)}}
\newcommand{\gammaa}{\gamma^{(a)}}
\newcommand{\Wa}{W^{(a)}}
\newcommand{\bigWNa}{\mathbb{W}_N^{(a)}}
\newcommand{\ddz}{\ddx{z}}

\newcommand{\ppphi}{\ppx{\varphi}}
\newcommand{\rhoppphi}{\rho \ppphi}

\newcommand{\pptheta}{\ppx{\theta}}
\newcommand{\ppthetatwo}{{\partial^2 \over \partial \theta^2}}
\newcommand{\bigWNi}{\mathbb{W}_N^{(1)}}

\newcommand{\bigscopi}{\bigscop^{(1)}}
\newcommand{\bigscopNi}{\bigscopi_{N}}

\newcommand{\bigscopo}{\bigscop^{(0)}}
\newcommand{\bigscopNo}{\bigscopo_{N}}

\newcommand{\href}{\mathcal{H}}

\usepackage[margin=10pt,font=footnotesize,labelfont=bf]{caption}
\usepackage{subcaption}
\usepackage{float}
\usepackage{amssymb}
\usepackage{amsmath}
\usepackage{bm}
\usepackage{bbm}
\usepackage{mleftright}
\usepackage{todonotes}
\usepackage{csquotes}

\usepackage[square,numbers]{natbib}
\usepackage{url}
\usepackage{amsthm}

\newtheorem{proposition}{Proposition}
\newtheorem{lemma}{Lemma} 
\newtheorem{theorem}{Theorem} 
\newtheorem{definition}{Definition}

\newtheorem{corollary}{Corollary}

\def\bsrefeqn#1{equation (\ref{#1})}

\def\bsreflemma#1{Lemma \ref{#1}}

\def\bsrefdef#1{Definition \ref{#1}}

\def\addtab#1={#1\;&=}

  \def\leqaddtab#1\leq{#1\;&\leq}

\def\pr(#1){\left({#1}\right)}
\def\br[#1]{\left[{#1}\right]}
\def\fbr[#1]{\!\left[{#1}\right]}

\def\ip<#1>{\left\langle{#1}\right\rangle}
\def\iip<#1>{\left\langle\!\langle{#1}\right\rangle\!\rangle}

\def\norm#1{\left\| #1 \right\|}

\def\fpr(#1){\!\pr({#1})}

\def\ceil#1{\left\lceil#1\right\rceil}

\def\mapengine#1,#2.{\mapfunction{#1}\ifx\void#2\else\mapengine #2.\fi }

\def\map[#1]{\mapengine #1,\void.}

\def\mapenginesep_#1#2,#3.{\mapfunction{#2}\ifx\void#3\else#1\mapengine #3.\fi }

\def\mapsep_#1[#2]{\mapenginesep_{#1}#2,\void.}

\def\vcbr[#1]{\pr(#1)}

\def\bvect[#1,#2]{
{
\def\dots{\cdots}
\def\mapfunction##1{\ | \  ##1}
	\sopmatrix{
		 \,#1\map[#2]\,
	}
}
}

\def\vect[#1]{
{\def\dots{\ldots}
	\vcbr[{#1}]
}}

\def\vectt[#1]{
{\def\dots{\ldots}
	\vect[{#1}]^{\top}
}}

\def\Vectt[#1]{
{
\def\mapfunction##1{##1 \cr} 
\def\dots{\vdots}
	\begin{pmatrix}
		\map[#1]
	\end{pmatrix}
}}

\def\R{{\mathbb R}}

\def\D{{\rm d}}

\def\tF_#1{{\tt F}_{#1}}

\def\tFC_#1{{\tt T}_{#1}}

\def\secref#1{Section~\ref{Section:#1}}

\def\qfor{\quad\hbox{for}\quad}

\def\elllRpz_#1{\ell_{#1{\rm z}}^{(\lambda,R),p}}

\def\sopmatrix#1{\begin{pmatrix}#1\end{pmatrix}}

\def\Problem#1#2\par{\begin{problem}\label{pb:#1} #2\end{problem}}
\def\Theorem#1#2\par{\begin{theorem}\label{th:#1} #2\end{theorem}}
\def\Conjecture#1#2\par{\begin{conjecture}\label{conj:#1} #2\end{conjecture}}
\def\Proposition#1#2\par{\begin{proposition}\label{prop:#1} #2\end{proposition}}
\def\Definition#1#2\par{\begin{definition}\label{def:#1} #2\end{definition}}
\def\Corollary#1#2\par{\begin{corollary}\label{cr:#1} #2\end{corollary}}
\def\Lemma#1#2\par{\begin{lemma}\label{lm:#1} #2\end{lemma}}
\def\Example#1#2\par{\begin{example}\label{ex:#1} #2\end{example}}
\def\Remark #1\par{\begin{remark*}#1\end{remark*}}

\def\Proof{\begin{proof}}
\def\mqed{\end{proof}}

\def\Figuretwow[#1,#2]#3#4\par{
\begin{figure}[tb]
\begin{center}{
\includegraphics[width=#3]{Figures/#1}\includegraphics[width=#3]{Figures/#2}}
\end{center}
\caption{#4}\label{fig:#1} 
\end{figure}
}

\title{Sparse spectral methods for partial differential equations on spherical caps}
\author{Ben Snowball, Sheehan Olver\thanks{Department of Mathematics, Imperial College, London, England}}

\begin{document}

\maketitle

\begin{abstract}
In recent years, sparse spectral methods for solving partial differential equations have been derived using hierarchies of classical orthogonal polynomials on intervals, disks, disk-slices and triangles. In this work we extend the methodology to a hierarchy of non-classical multivariate orthogonal polynomials on spherical caps. The entries of discretisations of partial differential operators can be effectively computed using formulae in terms of (non-classical) univariate orthogonal polynomials. We demonstrate the results on partial differential equations involving the spherical Laplacian and biharmonic operators, showing spectral convergence. 
\end{abstract}

\section{Introduction}

This paper develops sparse spectral methods for solving linear partial differential equations on certain subsets of the sphere---specifically spherical caps. More precisely, we consider the solution of partial differential equations on the spherical cap $\Omega$ defined by
\begin{align*}
	\Omega := \{(x,y,z) \in \R^3 \quad | \quad \alpha < z < \beta, \: x^2 + y^2 + z^2 = 1\}
\end{align*}
where $\alpha \in (-1,1)$ and $\beta := 1$.

\begin{remark}
For simplicity we focus on the case of a spherical cap, though there is an extension to a spherical band by taking $\beta \in (\alpha,1)$. The methods presented here translate to the spherical band case by including the necessary adjustments to the weights and recurrence relations we present in this paper. These adjustments make the mathematics more involved, which is why they are omitted here, but the approach is the same.
\end{remark}

We advocate using a basis that is polynomial in cartesian coordinates, that is, polynomial in $x$, $y$, and $z$, and orthogonal with respect to a prescribed weight: that is, multivariate orthogonal polynomials, whose construction was considered in \cite{olver2020orthogonal}. Equivalently, we can think of these as polynomials modulo the vanishing ideal $\{ x^2 + y^2 + z^2 = 1 \}$, or simply as a linear recombination of spherical harmonics that are orthogonalised on a subset of the sphere. This is in contrast to more standard approaches based on mapping the geometry to a simpler one (e.g., a rectangle or disk) and using orthogonal polynomials in the mapped coordinates (e.g., a basis that is polynomial in the spherical coordinates $\varphi$ and $\theta$). The benefit of the new approach is that we do not need to resolve Jacobians, and thereby we can achieve sparse discretisations for partial differential operators, including those with polynomial variable coefficients. Further, we avoid the singular nature at the poles or  as $\alpha$ approaches $0$ that  such a projection may give, since our new approach yields a smooth polynomial basis for all $\alpha \in [-1,1)$. 

On the spherical cap, the family of weights we consider are of the form
\begin{align*}
	\Wa(x,y,z) := (z - \alpha)^a, \qfor (x,y,z) \in \Omega,
\end{align*}
noting that $\Wa(x,y,z) = 0$ for $(x,y,z) \in \partial \Omega$ when $a > 0$. The corresponding OPs denoted $\scopnkia(x,y,z)$, where $n$ denotes the polynomial degree, $0 \le k \le n$ and $i \in \{0, \min(1,k)\}$. We define these to be orthogonalised lexicographically, that is,
\begin{align*}
	\scopnkia(x,y,z) = C_{n,k,i} \: x^{k-i} \: y^i \: z^{n-k} + (\hbox{lower order terms})
\end{align*}
where $C_{n,k,i} \neq 0$ and \enquote{lower order terms} includes degree $n$ polynomials of the form $x^{j - i} \: y^{i} \: z^{n-j}$ where $j < k$. The precise normalization arises from their definition in terms of one-dimensional OPs in Definition~\ref{def:OPconstruction}.

We consider partial differential operators involving the spherical Laplacian (the Laplace--Beltrami operator): in spherical coordinates 
\begin{align*}
	z &= \cos\varphi, \\
	x &= \sin\varphi \cos\theta = \rho(z) \cos\theta, \\
	y &= \sin\varphi \sin\theta = \rho(z) \sin\theta.
\end{align*}
where $ \rho(z) := \sqrt{1-z^2}$, we have
\begin{align*}
	\DeltaS &= {1 \over \sinphi} \ppphi \Big( \sinphi \ppphi \Big) + {1 \over \sin^2 \varphi} \ppthetatwo = {1 \over \rho} \ppphi \Big( \rho \ppphi \Big) + {1 \over \rho^2} \ppthetatwo
\end{align*}
i.e. $\DeltaS f(\xvec) = \Delta f({\xvec \over \norm{\xvec}})$ for $\xvec := (x,y,z) \in \R^3$. We do so by considering the component operators $\rho {\partial \over \partial \varphi}$ and ${\partial \over \partial \theta}$ applied to OPs with a specific choices of weight so that their discretisation is sparse, see Theorem~\ref{theorem:sparsityofdifferentialoperators}.  Sparsity comes from expanding the domain and range of an operator using different choices of the parameter $a$, a la the ultraspherical spectral method for intervals \cite{olver2013fast}, triangles \cite{olver2019triangle} and disk-slices and trapeziums \cite{snowball2019sparse}, and the related work on sparse discretisations on disks \cite{vasil2016tensor} and spheres \cite{vasil2019tensor,lecoanet2019tensor}.  As in the  disk-slice case in 2D \cite{snowball2019sparse}, we use an integration-by-parts argument to deduce the sparsity structure.

 The three-dimensional orthogonal polynomials defined here involve the same non-classical (in fact, semi-classical \cite[\S5]{magnus1995painleve}) 1D OPs as those outlined for the disk-slice, and so methods for calculating these 1D OP recurrence coefficients and integrals have already been outlined \cite{snowball2019sparse}. In particular, by exploiting the connection with these 1D OPs we can construct discretizations of general partial differential operators of size $(p+1)^2 \times (p+1)^2$ in $O(p^3)$ operations, where $p$ is the total polynomial degree. This clearly compares favourably to proceeding in a na\"ive approach where one would require $O(p^6)$ operations.

Note that we consider partial differential operators that are not necessarily rotational invariant:  for example, one can use these techniques for  Schr\"odinger operators $\DeltaS + v(x,y,z)$ where $v$ is first approximated by a polynomial. A nice feature though is that if the partial differential operator is invariant with respect to rotation around the $z$ axis (e.g., a Schr\"odinger operator with potential $v(z)$) the discretisation decouples, and can be reordered as a block-diagonal matrix. This improves the complexity further to an optimal $O(p^2$), which is demonstrated in Figure \ref{fig:complexity} with $v(x,y,z) = \cos z$.

An overview of the paper is as follows:  

\noindent \secref{OPs}: We present our definition of a (one-parameter) family of 3D orthogonal polynomials (OPs) on the spherical cap domain $\Omega$, by combining 1D OPs on the interval $(\alpha, 1)$ with Chebyshev polynomials, to form 3D OPs on the spherical cap surface. We show that these families will lead to sparse Jacobi operators for multiplication by $x, y, z$ and demonstrate how to obtain the 3D OPs.

\noindent\secref{PDOs}: We define several partial differential operators such as spherical Laplacians and  show that these will  be sparse when applied to a suitable choice of expansions in bases built from OPs on the spherical cap. We can exactly calculate the non-zero entries of these sparse operators using the quadrature rule associated with the non-classical 1D OPs.

\noindent\secref{Computation}: We derive a quadrature rule on the spherical cap that can be used to expand a function in the OP basis up to a given order $N$, and demonstrate how to evaluate a function using the Clenshaw algorithm using the coefficients of its expansion.

\noindent\secref{Examples}: We demonstrate the proposed technique for solving differential equations on the spherical cap such as the Poisson equation, variable coefficient Helmholtz equation, and Biharmonic equation.  


\noindent{\bf Acknowledgments}: The first author was supported by the Engineering and Physical Sciences Research Council Mathematics of Planet Earth Centre for Doctoral Training at Imperial College London and the University of Reading, with grant number EP/L016613/1.The second author was supported by the Leverhulme Trust Research Project Grant RPG-2019-144 ``Constructive approximation theory on and inside algebraic curves and surfaces''.

\section{Orthogonal polynomials on spherical caps}\label{Section:OPs}

In this section we outline the construction and some basic properties of $\scopnkia(x,y,z)$.

\subsection{Explicit construction}

We can construct the 3D orthogonal polynomials on $\Omega$ from 1D orthogonal polynomials on the interval \([\alpha,\beta]\), and from Chebyshev polynomials. We do so in terms of Fourier series, which, following \cite{olver2020orthogonal}, we write here as orthogonal polynomials in $x$ and $y$:

\begin{definition}\label{def:Ydefinition}
	Define the unit circle $\omega := \{ \mathbf{x} = (x,y) \in \R^2 \: | \: x^2 + y^2 = 1\}$, and define the parameter $\theta$ for each $(x,y) \in \omega$ by $x = \cos\theta$, $y = \sin\theta$. Define the polynomials $\{\chki\}$ for $k = 0, 1, \dots$, $i = 0, 1$ on $(x,y) \in \omega$ by
\begin{align*}
	\ch_{0,0}(\xvec) \equiv \ch_{0,0}(x, y) &:= \ch_{0} =: \ch_{0,0}(\theta) \\
	\ch_{k,0}(\xvec) \equiv \ch_{k,0}(x, y) \ &:= T_k(x) = \cos k \theta =: \ch_{k,0}(\theta), \quad k = 1,2,3,\dots \\
	\ch_{k,1}(\xvec) \equiv \ch_{k,1}(x, y) &:= y \: U_{k-1}(x) = \sin k \theta =: \ch_{k,1}(\theta), \quad k = 1,2,3,\dots
\end{align*}
where $\ch_{0} := \frac{\sqrt{2}}{2}$ and $T_k$, $U_{k-1}$ are the standard Chebyshev polynomials on the interval $[-1,1]$. The $\{Y_{k,i}\}$ are orthonormal with respect to the inner product
\begin{align*}
	\ip< p, \: q >_{\ch} &:= \frac{1}{\pi} \: \int_0^{2\pi} p(\xvec(\theta)) \: q(\xvec(\theta)) \: \D \theta
\end{align*}
\end{definition}
Note that we have defined $\ch_0$ so as to ensure orthonormality. 

\begin{proposition}[\cite{olver2020orthogonal}]\label{prop:construction}
	Let $w : (\alpha,\beta) \: \to \R$ be a weight function. For $n = 0,1,2,\dots, $ let $\{r_{n,k}\}$ be polynomials orthogonal with respect to the weight $\rho(x)^{2k} w(x)$ where $0 \le k \le n$. Then the 3D polynomials defined on $\Omega$
\begin{align*}
	\scopnki(x,y,z) := r_{n-k,k}(z) \: \rho(z)^k \: \ch_{k,i}\fpr({x \over \rho(z)}, {y \over \rho(z)})
\end{align*}
for $i \in {0,1}, \: 0 \le k \le n, \: n = 0,1,2,\dots$ are orthogonal polynomials with respect to the inner product
\begin{align*}
	\ip< p, \: q > &:= \int_\Omega p(x,y,z) \: q(x,y,z) \: w(z) \: \D A \\
	&= \int_0^{\cos^{-1}(\alpha)} \int_0^{2\pi} p\big(\sinphi \costheta, \sinphi \sintheta, \cosphi \big) \: q\big(\sinphi \costheta, \sinphi \sintheta, \cosphi \big) \: \times \\
	&\qquad w(\cosphi) \: \sinphi \: \D \theta \: \D \varphi \\
	&= \int_\alpha^1 \int_0^{2\pi} p\big(\rho(z) \cos\theta, \rho(z) \sin\theta, z\big) \: q\big(\rho(z) \cos\theta, \rho(z) \sin \theta, z\big) \: w(z) \: \D \theta \: \D z
\end{align*}
on $\Omega$, where $\D A = \sinphi \: \D \theta \: \D \varphi$ is the uniform spherical measure on $\Omega$. 
\end{proposition}

For the spherical cap, we can use Proposition \ref{prop:construction} to create our one-parameter family of OPs. We first introduce notation for our family of non-classical univariate OPs that will be used as the $r_n$ polynomials above.
\begin{definition}[\cite{snowball2019sparse}]\label{def:OPconstruction}
	Let $\genjacw^{(a,b)}(x)$ be a weight function on the interval $(\alpha, 1)$ given by:
\begin{align*}
	\genjacw^{(a,b)}(x) &:= (x - \alpha)^{a} \: \rho(x)^{b}
\end{align*}
and define the associated inner product by:
\begin{align}
	\ip< p, \: q >_{\genjacw^{(a,b)}} &:= \frac{1}{\normgenjac^{(a,b)}} \: \int_\alpha^1 p(x) \: q(x) \: \genjacw^{(a,b)}(x) \: \D x \label{eqn:ipgenjac}
\end{align}
where
\begin{align}
	\normgenjac^{(a,b)} := \int_\alpha^1 \: \genjacw^{(a,b)}(x) \: \D x \label{eqn:ipnormalisation}
\end{align}
is a normalising constant.
Denote the two-parameter family of orthonormal polynomials on $[\alpha,\beta]$ by $\{\genjac_n^{(a,b)}\}$, orthonormal with respect to the inner product defined in (\ref{eqn:ipgenjac}).
\end{definition}

We can now define the 3D OPs for the spherical cap.
\begin{definition}\label{def:constuction}
	Define the one-parameter 3D orthogonal polynomials via:
\begin{align}
	\scopnkia(x,y,z) := \genjacnmk^{(a,2k)}(z) \: \rho(z)^k \: \chki\fpr(\frac{x}{\rho(z)}, \frac{y}{\rho(z)}), \quad (x,y,z) \in \Omega.
\end{align}
\end{definition}

By construction, $\{\scopnkia\}$ are orthogonal with respect to the inner product
\begin{align*}
	\ip< p, \: q >_{\scopa} &:= \int_\Omega p(\xvec, z) \: q(\xvec, z) \: \genjacw^{(a,0)}(z) \: \D A \\
	&= \int_\alpha^1 \int_0^{2\pi} p(\rho(z) \cos\theta, \rho(z) \sin\theta, z) \: q(\rho(z) \cos\theta, \rho(z) \sin\theta, z) \: \D \theta \: \genjacw^{(a,0)}(z) \: \D z,
\end{align*}
with
\begin{align}
	\norm{\scopnkia}_{\scopa}^2 &:= \ip< \scopnkia, \: \scopnkia >_{\scopa} = \pi \: \normgenjac^{(a, 2k)}. \label{eqn:normscop}
\end{align}
We note that the weight $\genjacw^{(a,b)}(z)$ has been used in the construction of 2D orthogonal polynomials on disk-slices and trapeziums \cite{snowball2019sparse}, where a method for obtaining recurrence coefficients and evaluating integrals was established (the weight is in fact semi-classical, and is equivalent to a generalized Jacobi weight \cite[\S5]{magnus1995painleve}).

\subsection{Jacobi matrices}\label{subsection:jacobimats}

We can express the three-term recurrences associated with $\genjac_n^{(a,b)}$ as
\begin{align}
	x \genjac_n^{(a,b)}(x) &= \beta_n^{(a,b)} \genjac_{n+1}^{(a,b)}(x) + \alpha_n^{(a,b)} \genjac_n^{(a,b)}(x) + \beta_{n-1}^{(a,b)} \genjac_{n-1}^{(a,b)}(x) \label{eqn:Rrecurrence}
\end{align}
where the coefficients are calculatable (see \cite{snowball2019sparse}). We can use (\ref{eqn:Rrecurrence}) to determine the 3D recurrences for $\scopnkia(x,y,z)$. Importantly, we can deduce sparsity in the recurrence relationships.  We first require the following lemma.

\begin{lemma}\label{lemma:Yrecurrence} 
The following identities hold for $k = 2,3,\dots$, $j = 0,1,\dots$ and $i, h \in \{0,1\}$:
\begin{align*}
	&1) \quad \int_0^{2\pi} \ch_0 \: \chjh(\theta) \: \cos\theta \: \D \theta = \ch_0 \: \pi \: \delta_{0,h} \: \delta_{1,j} \\
	&2) \quad \int_0^{2\pi} \ch_0 \: \chjh(\theta) \: \sin\theta \: \D \theta = \ch_0 \: \pi \: \delta_{1,h} \: \delta_{1,j} \\
	&3) \quad \int_0^{2\pi} \ch_{1,i}(\theta) \: \chjh(\theta) \: \cos\theta \: \D \theta = \pi \: \delta_{i,h} \: (\ch_0 \: \delta_{0,j} + \half \delta_{2,j}) \\
	&4) \quad \int_0^{2\pi} \ch_{1,i}(\theta) \: \chjh(\theta) \: \sin\theta \: \D \theta = \pi \:  \delta_{|i-1|,h} \: ((-1)^{i+1} \: \ch_0 \: \delta_{0,j} + (-1)^i \: \half \: \delta_{2,j}) \\
	&5) \quad \int_0^{2\pi} \chki(\theta) \: \chjh(\theta) \: \cos\theta \: \D \theta = \half \: \pi \: \delta_{i,h} \: (\delta_{k-1,j} + \delta_{k+1,j}) \\
	&6) \quad \int_0^{2\pi} \chki(\theta) \: \chjh(\theta) \: \sin\theta \: \D \theta = \half \: \pi \: \delta_{|i-1|,h} \: ((-1)^{i+1} \: \delta_{k-1,j} + (-1)^i \: \delta_{k+1,j}).
\end{align*}
\end{lemma}

\begin{proof}
Each follows from the definitions of $\ch_{k,i}$ and $\ch_0$, as well as the relationships:
\begin{align*}
	2 \cos k \theta \cos\theta &= \cos (k-1)\theta + \cos(k+1)\theta \\
	2 \sin k \theta \cos\theta &= \sin (k-1)\theta + \sin(k+1)\theta \\
	2 \cos k \theta \sin\theta &= - \sin (k-1)\theta + \sin(k+1)\theta \\
	2 \sin k \theta \sin\theta &= \cos (k-1)\theta - \cos(k+1)\theta.
\end{align*}
\end{proof}

\begin{lemma}\label{lemma:Qrecurrence} 
Define
\begin{align}
	\eta_{k} :=
		\begin{cases}
			0 &\text{if } k < 0 \\
			\ch_0 &\text{if } k = 0 \\
			\half &\text{otherwise}
		\end{cases}
\end{align}
$\scopnkia(x,y,z)$ satisfy the following recurrences:
\begin{align*}
	x \: \scopnkia(x,y,z) &= \alphaa_{n,k,1} \:  \scopa_{n-1, k-1, i}(x, y, z) + \alphaa_{n,k,2} \:  \scopa_{n-1, k+1, i}(x, y, z) \nonumber \\
		& \quad \quad + \alphaa_{n,k,3} \:  \scopa_{n, k-1, i}(x, y, z) + \alphaa_{n,k,4} \:  \scopa_{n, k+1, i}(x, y, z) \nonumber \\
		& \quad \quad + \alphaa_{n,k,5} \:  \scopa_{n+1, k-1, i}(x, y, z) + \alphaa_{n,k,6} \:  \scopa_{n+1, k+1, i}(x, y, z), \\ \\
	y \: \scopnkia(x,y,z) &= \betaa_{n,k,i,1} \:  \scopa_{n-1, k-1, |i-1|}(x, y, z) + \betaa_{n,k,i,2} \:  \scopa_{n-1, k+1, |i-1|}(x, y, z) \nonumber \\
		& \quad \quad + \betaa_{n,k,i,3} \:  \scopa_{n, k-1, |i-1|}(x, y, z) + \betaa_{n,k,i,4} \:  \scopa_{n, k+1, |i-1|}(x, y, z) \nonumber \\
		& \quad \quad + \betaa_{n,k,i,5} \:  \scopa_{n+1, k-1, |i-1|}(x, y, z) + \betaa_{n,k,i,6} \:  \scopa_{n+1, k+1, |i-1|}(x, y, z), \\ \\
	z \: \scopnkia(x,y,z) &= \gammaa_{n,k,1} \: \scopa_{n-1, k, i}(x, y, z) + \gammaa_{n,k,2} \: \scopa_{n, k, i}(x, y, z) + \gammaa_{n,k,3} \: \scopa_{n+1, k, i}(x, y, z),
\end{align*}
for $(x,y,z) \in \Omega$, where
\begin{align*}
	\alphaa_{n,k,1} &:= \eta_{k-1} \: \ip<\genjacnmk^{(a, 2k)}, \: \genjacnmk^{(a, 2(k-1))}>_{\genjacw^{(a, 2k)}}, \\
	\alphaa_{n,k,2} &:= \eta_{k} \: \ip<\genjacnmk^{(a, 2k)}, \genjac_{n-k-2}^{(a, 2(k+1))}>_{\genjacw^{(a, 2(k+1))}}, \\
	\alphaa_{n,k,3} &:=\eta_{k-1} \: \ip<\genjacnmk^{(a, 2k)}, \: \genjac_{n-k+1}^{(a, 2(k-1))}>_{\genjacw^{(a, 2k)}}, \\
	\alphaa_{n,k,4} &:= \eta_{k} \: \ip<\genjacnmk^{(a, 2k)}, \genjac_{n-k-1}^{(a, 2(k+1))}>_{\genjacw^{(a, 2(k+1))}}, \\
	\alphaa_{n,k,5} &:= \eta_{k-1} \: \ip<\genjacnmk^{(a, 2k)}, \: \genjac_{n-k+2}^{(a, 2(k-1))}>_{\genjacw^{(a, 2k)}}, \\
	\alphaa_{n,k,6} &:=\eta_{k} \: \ip<\genjacnmk^{(a, 2k)}, \genjacnmk^{(a, 2(k+1))}>_{\genjacw^{(a, 2(k+1))}}, \\
	\betaa_{n,k,i,j} &:= 
		\begin{cases}
			- \alphaa_{n,k,j} \quad &\text{if }(i = 0 \text{ and } j \text{ is odd}) \text{ or }(i = 1 \text{ and } j \text{ is even}) \\
			\alphaa_{n,k,j} \quad &\text{otherwise}
		\end{cases}, \\	
	\gammaa_{n,k,1} &:= \beta_{n-k-1}^{(a, 2k)}, \qquad \gammaa_{n,k,2} := \alpha_{n-k}^{(a, 2k)}, \qquad \gammaa_{n,k,3} := \beta_{n-k}^{(a, 2k)}.
\end{align*}
\end{lemma}
\begin{remark}
For $z$ multiplication, note that different Fourier modes do not interact. This is because $z$ is rotationally invariant. 
\end{remark}
\begin{proof}
The 3-term recurrence for multiplication by $z$ follows from (\ref{eqn:Rrecurrence}). For the recurrence for multiplication by $x$, since $\{\scopmjha\}$ for $m = 0,\dots,n+1$, $j = 0,\dots,m$, $h = 0,1$ is an orthogonal basis for any degree $n+1$ polynomial on $\Omega$, we can expand 
\begin{align*}
	x \: \scopnkia(x,y,z) = \sum_{m=0}^{n+1} \sum_{j=0}^m \sum_{h=0}^1 c_{m,j} \: \scopmjha(x,y,z).
\end{align*}
These coefficients are given by
\begin{align*}
	c_{m,j} = {\ip< x \: \scopnkia, \scopmjha >_{\scopa}}{\norm{\scopmjha}^{-2}_{\scopa}}
\end{align*}
where we show the non-zero coefficients that result are the $\alphaa_{n,k,1},\dots,\alphaa_{n,k,6}$ in the lemma.
Recall from equation (\ref{eqn:normscop}) that $\norm{\scopmjha}_{\scopa}^2 = \pi \: \normgenjac^{(a,2j)}$. Then for $m = 0,\dots,n+1$, $j = 0,\dots,m$, using a change of variables $(\cos\theta \sin\varphi, \: \sin\theta\sin\varphi, \: \cos\varphi) = (x, y, z)$:
\begin{align*}
	&\ip<x \: \scopnkia, \scopmjha>_{\scopa} \\
	&= \int_\Omega \scopnkia(\xvec,z) \: \scopmjha(\xvec,z) \: x \: \genjacw^{(a,0)}(z) \: \D A \\
	&= \Big( \int^1_\alpha \genjacnmk^{(a, 2k)}(z) \: \genjacmmj^{(a, 2j)}(z) \: \rho(z)^{k+j+1} \: \genjacw^{(a, 0)}(z) \: \D z \Big) \cdot \: \Big( \int_0^{2\pi} \chki(\theta) \: \chjh(\theta) \: \cos\theta \: \D \theta \Big) \\
	&= \Big( \int^1_\alpha \genjacnmk^{(a, 2k)}(z) \: \genjacmmj^{(a, 2j)}(z) \: \genjacw^{(a, k + j + 1)}(z) \: \D z \Big) \cdot \: \Big( \int_0^{2\pi} \chki(\theta) \: \chjh(\theta) \: \cos\theta \: \D \theta \Big) \\
	&= \half \: \pi \: \delta_{i,h} \: (\eta_{k-1} \: \delta_{k-1, j} + \eta_{k} \: \delta_{k+1, j})  \int^1_\alpha \genjacnmk^{(a, 2k)}(z) \: \genjacmmj^{(a, 2j)}(z) \: \genjacw^{(a, k + j + 1)}(z) \: \D z.
\end{align*}
where $\delta_{k, j}$ is the standard Kronecker delta function, using Lemma \ref{lemma:Yrecurrence}. Similarly, for the recurrence for multiplication by $y$, we can expand 
\begin{align*}
	y \: \scopnkia(x,y,z) = \sum_{m=0}^{n+1} \sum_{j=0}^m \sum_{h=0}^1 d_{m,j} \: \scopmjha(x,y,z).
\end{align*}
These coefficients are given by
\begin{align*}
	d_{m,j} = {\ip< y \: \scopnkia, \scopmjha >_{\scopa}}{\norm{\scopmjha}^{-2}_{\scopa}}
\end{align*}
where we show the non-zero coefficients that result are the $\betaa_{n,k,1},\dots,\betaa_{n,k,6}$ in the lemma:
\begin{align*}
	&\ip<y \: \scopnkia, \scopmjha>_{\scopa} \\
	&= \int_\Omega \scopnkia(\xvec,z) \: \scopmjha(\xvec,z) \: y \: \genjacw^{(a,0)}(z) \: \D A \\
	&= \Big( \int^1_\alpha \genjacnmk^{(a, 2k)}(z) \: \genjacmmj^{(a, 2j)}(z) \: \rho(z)^{k+j+1} \: \genjacw^{(a, 0)}(z) \: \D z \Big) \cdot \: \Big( \int_0^{2\pi} \chki(\theta) \: \chjh(\theta) \: \sin\theta \: \D \theta \Big) \\
	&= \Big( \int^1_\alpha \genjacnmk^{(a, 2k)}(z) \: \genjacmmj^{(a, 2j)}(z) \: \genjacw^{(a, k + j + 1)}(z) \: \D z \Big) \cdot \: \Big( \int_0^{2\pi} \chki(\theta) \: \chjh(\theta) \: \sin\theta \: \D \theta \Big) \\
	&= \half \: \pi \: \delta_{|i-1|,h} \: \big[(-1)^{i+1} \: \eta_{k-1} \: \delta_{k-1, j} + (-1)^i \: \eta_{k} \: \delta_{k+1, j}\big] \int^1_\alpha \genjacnmk^{(a, 2k)}(z) \: \genjacmmj^{(a, 2j)}(z) \: \genjacw^{(a, k + j + 1)}(z) \: \D z.
\end{align*}
where again $\delta_{k, j}$ is the standard Kronecker delta function, and we have used Lemma \ref{lemma:Yrecurrence}.

\end{proof}

The recurrences in Lemma \ref{lemma:Qrecurrence} lead to Jacobi operators that correspond to multiplication by $x$, $y$ and $z$. In later sections we will use an ordering of the OPs so that they are grouped by Fourier mode $k$, which is convenient for the application of differential and other operators to the vector of coefficients of a given function's expansion (some operators will exploit this ordering for operators where Fourier modes do not interact, and thus will be block-diagonal). Before that though, the ordering we will use in the remainder of this section is convenient for establishing Jacobi operators for multiplication by $x$, $y$ and $z$, and hence building the OPs and importantly obtaining the associated \textit{recurrence coefficient matrices} necessary for efficient function evaluation using the Clenshaw algorithm. In practice, it is simply a matter of converting coefficients between the two orderings. To this end, we define our OP-building ordering as follows. For $n=0,1,2,\dots$:
\begin{align*}
	\bigscoptna := 
		\begin{pmatrix}
			\scopa_{n,0,0}(x,y,z) \\
			\scopa_{n,1,0}(x,y,z) \\
			\scopa_{n,1,1}(x,y,z) \\
			\vdots \\
			\scopa_{n,n,0}(x,y,z) \\
			\scopa_{n,n,1}(x,y,z)
		\end{pmatrix} \in \R^{2n+1}, 
	\quad \quad 
	\bigscopta := 
		\begin{pmatrix}
			\bigscopta_0 \\
			\bigscopta_1 \\
			\bigscopta_2 \\
			\vdots \\
		\end{pmatrix}
\end{align*}
and set $ J_x^{(a)},  J_y^{(a)},  J_z^{(a)}$ as the Jacobi matrices corresponding to
\begin{align}
	J_x^{(a)} \: \bigscopta(x,y,z) = x \: \bigscopta(x,y,z), \nonumber \\
	J_y^{(a)} \: \bigscopta(x,y,z) = y \: \bigscopta(x,y,z), \label{eqn:jacobimatricesdefinitionalt} \\
	J_z^{(a)} \: \bigscopta(x,y,z) = z \: \bigscopta(x,y,z). \nonumber
\end{align}
where
\begin{align*}
	J_{x/y/z}^{(a)} &= 
		\begin{pmatrix}
			B^{(a)}_{x/y/z, 0} & A^{(a)}_{x/y/z, 0} & & & & \\
			C^{(a)}_{x/y/z, 1} & B^{(a)}_{x/y/z, 1} & A^{(a)}_{x/y/z, 1} & & & \\
			& C^{(a)}_{x/y/z, 2} & B^{(a)}_{x/y/z, 2} & A^{(a)}_{x/y/z, 2} & & & \\
			& & C^{(a)}_{x/y/z, 3} & \ddots & \ddots & \\
			& & & \ddots & \ddots & \ddots \\
		\end{pmatrix}.
\end{align*}
Note that $J_x^{(a)}, J_y^{(a)}, J_z^{(a)}$ are \textit{banded-block-banded matrices}:

\begin{definition}
	A block matrix $A$ with blocks $A_{i,j}$ has block-bandwidths $(L,U)$ if $A_{i,j} = 0$ for $- L \leq j-i \leq U$, and sub-block-bandwidths $(\lambda, \mu)$ if all blocks $A_{i,j}$ are banded with bandwidths $(\lambda,\mu)$. A matrix where the block-bandwidths and sub-block-bandwidths are small compared to the dimensions is referred to as a banded-block-banded matrix. 
\end{definition}

Each of these Jacobi matrices are then block-tridiagonal (block-bandwidths $(1,1)$). For $J_x^{(a)}$, the sub-blocks have sub-block-bandwidths $(2,2)$:
\begin{align*}
	A^{(a)}_{x,n} &:= 
		\begin{pmatrix}
			0 & A^{(a)}_{n,0,6} & 0 & & \\
			A^{(a)}_{n,1,5} & \ddots & \ddots & & \\
			& \ddots & \ddots & \ddots & \\
			& & A^{(a)}_{n,n,5} & 0 & A^{(a)}_{n,n,6} \\
		\end{pmatrix} \in \R^{(2n+1)\times(2n+3)}, \quad n = 0,1,2,\dots \\
	B^{(a)}_{x,n} &:= 
		\begin{pmatrix}
			0 & A^{(a)}_{n,0,4} & & \\
			A^{(a)}_{n,1,3} & \ddots & \ddots & \\
			& \ddots & \ddots & A^{(a)}_{n,n-1,4} \\
			& & A^{(a)}_{n,n,3} & 0
		\end{pmatrix} \in \R^{(2n+1)\times(2n+1)}  \quad n = 0,1,2,\dots \\
	C^{(a)}_{x,n} &:= 
		\begin{pmatrix}
			0 & A^{(a)}_{n,0,2} & & \\
			A^{(a)}_{n,1,1} & \ddots & \ddots & \\
			& \ddots & \ddots &A^{(a)}_{n,n-2,2} \\
			& & \ddots & 0 \\
			& & & A^{(a)}_{n,n,1} \\
		\end{pmatrix} \in \R^{(2n+1)\times(2n-1)}, \quad n = 1,2,\dots
\end{align*}
where for $k = 1,\dots,N, \: n = k,\dots,N$
\begin{align*}
	A^{(a)}_{n,k,j} &:= 
		\begin{pmatrix}
			\alphaa_{n,k,j} & 0 \\
			0 & \alphaa_{n,k,j}
		\end{pmatrix} \in \R^{2\times2}, (k \ne 1 \text { for } j \text{ odd}) \\
	A^{(a)}_{n,0,j} &:=
		\begin{pmatrix}
			\alphaa_{n,0,j} & 0
		\end{pmatrix} \in \R^{1\times2}, j \text{ even} \\
	A^{(a)}_{n,1,j} &:=
		\begin{pmatrix}
			\alphaa_{n,1,j} \\
			0
		\end{pmatrix} \in \R^{2\times1}, j \text{ odd}.
\end{align*}
For $J_y^{(a)}$, the sub-blocks have sub-block-bandwidths $(3,3)$:
\begin{align*}
	A^{(a)}_{y,n} &:= 
		\begin{pmatrix}
			0 & B^{(a)}_{n,0,6} & 0 & & \\
			B^{(a)}_{n,1,5} & \ddots & \ddots & & \\
			& \ddots & \ddots & \ddots & \\
			& & B^{(a)}_{n,n,5} & 0 & B^{(a)}_{n,n,6} \\
		\end{pmatrix} \in \R^{(2n+1)\times(2n+3)}, \quad n = 0,1,2,\dots \\
	B^{(a)}_{y,n} &:= 
		\begin{pmatrix}
			0 & B^{(a)}_{n,0,4} & & \\
			B^{(a)}_{n,1,3} & \ddots & \ddots & \\
			& \ddots & \ddots & B^{(a)}_{n,n-1,4} \\
			& & B^{(a)}_{n,n,3} & 0
		\end{pmatrix} \in \R^{(2n+1)\times(2n+1)}  \quad n = 0,1,2,\dots \\
	C^{(a)}_{y,n} &:= 
		\begin{pmatrix}
			0 & B^{(a)}_{n,0,2} & & \\
			B^{(a)}_{n,1,1} & \ddots & \ddots & \\
			& \ddots & \ddots & B^{(a)}_{n,n-2,2} \\
			& & \ddots & 0 \\
			& & & B^{(a)}_{n,n,1} \\
		\end{pmatrix} \in \R^{(2n+1)\times(2n-1)}, \quad n = 1,2,\dots
\end{align*}
where for $k = 1,\dots,N, \: n = k,\dots,N$  
\begin{align*}
	B^{(a)}_{n,k,j} &:= 
		\begin{pmatrix}
			0 & \betaa_{n,k,0,j} \\
			\betaa_{n,k,1,j} & 0
		\end{pmatrix} \in \R^{2\times2}, (k \ne 1 \text { for } j \text{ odd}) \\
	B^{(a)}_{n,0,j} &:=
		\begin{pmatrix}
			0 & \betaa_{n,0,0,j}
		\end{pmatrix} \in \R^{1\times2}, j \text{ even} \\
	B^{(a)}_{n,1,j} &:=
		\begin{pmatrix}
			0 \\
			\betaa_{n,1,1,j}
		\end{pmatrix}\in \R^{2\times1}, j \text{ odd}.
\end{align*}
For $J_z^{(a)}$, the sub-blocks are diagonal, i.e. have sub-block-bandwidths $(0,0)$:
\begin{align*}
	A^{(a)}_{z,n} &:= 
		\begin{pmatrix}
			\Gamma^{(a)}_{n,0,3} & 0 & \\
			0 & \ddots & \ddots & & \\
			& \ddots & \ddots & \ddots & \\
			& & 0 & \Gamma^{(a)}_{n,n,3} & 0 \\
		\end{pmatrix} \in \R^{(2n+1)\times(2n+3)}, \quad n = 0,1,2,\dots \\
	B^{(a)}_{z,n} &:= 
		\begin{pmatrix}
			\Gamma^{(a)}_{n,0,2} & \\
			& \ddots & & \\
			& & \ddots & \\
			& & & \Gamma^{(a)}_{n,n,2}
		\end{pmatrix} \in \R^{(2n+1)\times(2n+1)}  \quad n = 0,1,2,\dots \\
	C^{(a)}_{z,n} &:= 
		\begin{pmatrix}
			\Gamma^{(a)}_{n,0,1} & 0 & & \\
			0 & \ddots & \ddots & \\
			& \ddots & \ddots & 0 \\
			& & \ddots & \Gamma^{(a)}_{n,n-1,1} \\
			& & & 0 \\
		\end{pmatrix} \in \R^{(2n+1)\times(2n-1)}, \quad n = 1,2,\dots
\end{align*}
where for $k = 1,\dots,N, \: n = k,\dots,N$
\begin{align}
	\Gamma^{(a)}_{n,k,j} &:= 
		\begin{pmatrix}
			\gamma_{n,k,j} & 0 \\
			0 & \gamma_{n,k,j}
		\end{pmatrix} \in \R^{2\times2}, \label{eqn:jacobisubblocksGamma1} \\
	\Gamma^{(a)}_{n,0,j} &:= \gammaa_{n,0,j}. \label{eqn:jacobisubblocksGamma2}
\end{align}
Note that the sparsity of the Jacobi matrices (in particular the sparsity of the sub-blocks) comes from the natural sparsity of the three-term recurrences of the 1D OPs and the circular harmonics, meaning that the sparsity is not limited to the specific spherical cap, and would extend to the spherical band.

\subsection{Building the OPs} 

Following the triangle case \cite{olver2019triangle}, we use a multivariate analogue of Clenshaw's algorithm for evaluation, where we  combine each system in (\ref{eqn:jacobimatricesdefinitionalt}) into a block-tridiagonal system, for any $(x,y,z) \in \Omega$:
\begin{align*}
	\renewcommand\arraystretch{1.3}
	\begin{pmatrix}
		1 & & & \\
		B_0-G_0(x,y,z) & A_0 & & \\
		C_1 & B_1-G_1(x,y,z) & \quad A_1 \quad & \\
		& C_2 & B_2 - G_2(x,y,z)  & \ddots \\
		& & \ddots &\ddots
	\end{pmatrix}
	\bigscopta(x,y,z) =
		\begin{pmatrix}
	 		\scopa_0 \\ 0 \\ 0 \\ 0 \\ \vdots  \\
		\end{pmatrix},
\end{align*}
where we note $\scopa_0 := \scopa_{0,0,0}(x,y,z) \equiv \genjac_0^{(a,0)} \: \ch_0$, and for each $n = 0,1,2\dots$,
\begin{align*}
	A_n &:= 
		\begin{pmatrix}
			A^{(a)}_{x,n} \\
			A^{(a)}_{y,n} \\
			A^{(a)}_{z,n}
		\end{pmatrix} \in \R^{3(2n+1)\times(2n+3)}, \quad
	C_n := 
		\begin{pmatrix}
			C^{(a)}_{x,n} \\
			C^{(a)}_{y,n} \\
			C^{(a)}_{z,n}
		\end{pmatrix} \in \R^{3(2n+1)\times(2n-1)} \quad (n \ne 0), \nonumber \\
	B_n &:= 
		\begin{pmatrix}
			B^{(a)}_{x,n} \\
			B^{(a)}_{y,n} \\
			B^{(a)}_{z,n}
		\end{pmatrix} \in \R^{3(2n+1)\times(2n+1)}, \quad
	G_n(x,y,z) := 
		\begin{pmatrix}
			xI_{2n+1} \\
			yI_{2n+1} \\
			zI_{2n+1}
		\end{pmatrix} \in \R^{3(2n+1)\times(n+1)}.
\end{align*}
 
For each $n = 0,1,2\dots$ let $\Dnt$ be any matrix that is a left inverse of $A_n$, i.e. such that $\Dnt A_n = I_{2n+3}$. Multiplying our system by the preconditioner matrix that is given by the block diagonal matrix of the $\Dnt$'s, we obtain a lower triangular system \cite[p78]{dunkl2014orthogonal}, which can be expanded to obtain the recurrence:
\begin{align*}
	\begin{cases}
		\bigscopta_{-1}(x,y,z) := 0 \\
		\bigscopta_{0}(x,y) := \scopa_0 \\
		\bigscopta_{n+1}(x,y) = -\Dnt (B_n-G_n(x,y,z)) \bigscopta_n(x,y,z) - \Dnt C_n  \, \bigscopta_{n-1}(x,y,z), \quad n = 0,1,2,\dots.
	\end{cases}
\end{align*}

Note that we can define an explicit $\Dnt$ as follows:
\begin{align*}
	\Dnt := 
		\begin{pmatrix}
			0 & & & 0 & & & (\Gamma^{(a)}_{n,0,3})^{-1} & & \\
			& \ddots & & & \ddots & & & \ddots \\
			& & 0 & & & 0 & & & (\Gamma^{(a)}_{n,n,3})^{-1}  \\
			& & & & & \bm{\eta}^\top_{0} & & & \\
			& & & & & \bm{\eta}^\top_{1} & & &
		\end{pmatrix} \in \R^{(2n+3)\times3(2n+1)},
\end{align*}
for $n = 1, 2, \dots$ where again $\Gamma^{(a)}_{n,k,3}$ are defined in equations (\ref{eqn:jacobisubblocksGamma1}, \ref{eqn:jacobisubblocksGamma2}) for $k=0,\dots,n$, and where $\bm{\eta}_{0}, \bm{\eta}_{1} \in \R^{3(2n+1)}$ with entries given by 
\begin{align*}
	\big(\bm{\eta}_{0}\big)_j &= 
		\begin{cases}
			\frac{1}{\betaa_{n,n,1,6}} & j = 2(2n+1) \\
			\frac{- \: \betaa_{n,n,1,5}}{\betaa_{n,n,1,6} \: \gammaa_{n, n-1, 3}} & j = 3(2n+1) - 3 \\
			0 & o/w
		\end{cases} \\
	\big(\bm{\eta}_{1}\big)_j &= 
		\begin{cases}
			\frac{1}{\alphaa_{n,n,6}} & j = 2n+1 \\
			\frac{- \: \alphaa_{n,n,5}}{\alphaa_{n,n,6} \: \gammaa_{n, n-1, 3}} & j = 3(2n+1) - 2  \text{ and } n > 1 \\
			0 & o/w
		\end{cases}
\end{align*}
For $n=0$, we can simply take
\begin{align*}
	D^\top_0 &:= 
		\begin{pmatrix}
			0 & 0 & \frac{1}{\gammaa_{0,0,3}} \\
			\frac{1}{\alphaa_{0,0,6}} & 0 & 0 \\
			0 & \frac{1}{\betaa_{0,0,6}} & 0
		\end{pmatrix} \in \R^{3\times3}.
\end{align*}

It follows that we can apply $\Dnt$ in $O(n)$ complexity, and thereby calculate $\bigscopta_{0}(x,y,z)$  through $\bigscopta_{n}(x,y,z)$ in optimal $O(n^2)$ complexity. 

\begin{definition}\label{def:clenshawmats}
	The \textit{recurrence coefficient matrices} associated with the OPs $\{\scopnkia\}$ are given by the matrices $A_n, B_n, C_n, \Dnt$ for $n = 0,1,2,\dots$ defined above.
\end{definition}

\section{Sparse partial differential operators}\label{Section:PDOs}

In this section we will derive the entries of spherical partial differential operators applied to our basis, demonstrating their sparsity in the process. To this end, as alluded to in Section~\ref{subsection:jacobimats}, we introduce new notation for a different ordering of the OP vector, in order to exploit the orthogonality the polynomials $\chki$ will bring and thus ensure the operators will be block-diagonal. Let $N \in \N$ and define:
\begin{align}
	\bigscopNka &:= 
		\begin{pmatrix}
			\scopa_{k,k,0}(x,y,z) \\
			\scopa_{k,k,1}(x,y,z) \\
			\vdots \\
			\scopa_{N,k,0}(x,y,z) \\
			\scopa_{N,k,1}(x,y,z)
		\end{pmatrix} \in \R^{2(N-k+1)},  \quad k = 1,\dots,N, \label{eqn:OPdefNka} \\ 
	\bigscopa_{N,0} &:= 
		\begin{pmatrix}
			\scopa_{0,0,0}(x,y,z) \\
			\vdots \\
			\scopa_{N,0,0}(x,y,z) \\
		\end{pmatrix} \in \R^{N+1}, \label{eqn:OPdefN0a} \\
	\bigscopNa &:= 
		\begin{pmatrix}
			\bigscopa_{N,0} \\
			\vdots \\
			\bigscopa_{N,N}
		\end{pmatrix} \in \R^{(N+1)^2} \label{eqn:OPdefNa}
\end{align}

We further denote the weighted set of OPs on $\Omega$ by 
\begin{align*}
	\bigWNa(x,y,z) := \genjacw^{(a,0)}(z) \: \bigscopNa(x,y,z),
\end{align*}

The operator matrices we derive here act on coefficient vectors, that represent a function $f(x,y,z)$ defined on $\Omega$ in spectral space -- such a function is approximated by its expansion up to degree $N$: 
\begin{align*}
	f(x,y,z) = \bigscopNa(x,y,z)^\top \vec f = \sum_{n=0}^{N} \sum_{k=0}^{n} \sum_{i=0}^{1} f_{n,k,i} \: \scopnkia(x,y,z),
\end{align*}
where $\vec f = (f_{n,k,i})$ is the coefficients vector for the function $f$.

\begin{definition}\label{def:differentialoperators}
	Let $a$ be a nonnegative parameter, and $\tilde a \ge 2$ be a positive integer. Define the operator matrices $D_\varphi^{(a)}, \: W_\varphi^{(a)}, \: D_\theta, \: \calL^{(a)\to(a+\tilde a)}, \: \calL_W^{(a)\to(a-\tilde a)}, \: \Delta^{(1)}_W$ according to:
\begin{align*}
	\rho \pfpx{\varphi}{f} (x,y,z)&= \bigscop_N^{(a+1)}(x,y,z)^\top \: D_\varphi^{(a)} \: \mathbf{f}, \\
	\rho \ppx{\varphi}[\genjacw^{(a,0)}(z) \: f(x,y,z)] &= \bigW_N^{(a-1)}(x,y)^\top \: W_\varphi^{(a)}\: \mathbf{f}, \\
	\pfpx{\theta}{f}(x,y,z) &= \bigscopNa(x,y,z)^\top \: D_\theta \: \mathbf{f}, \\
	\DeltaS f(x,y,z) &= \bigscopN^{(a+\tilde a)}(x,y,z)^\top \: \calL^{(a)\to(a+\tilde a)} \: \mathbf{f}, \\
	\DeltaS \big( \genjacw^{(a,0)}(z) \: f(x,y,z) \big) &= \bigW_N^{(a-\tilde a)}(x,y,z)^\top \: \calL_W^{(a)\to(a-\tilde a)} \: \mathbf{f},  \quad \text{ (for } a \ge 2 \text{ only)} \\
	\DeltaS \big( \genjacw^{(1,0)}(z) \: f(x,y,z) \big) &= \bigscopN^{(1)}(x,y,z)^\top \: \Delta^{(1)}_W \: \mathbf{f}, \quad \text{ (for } a = 1 \text{ only)}
\end{align*}
\end{definition}

The incrementing and decrementing of parameters as seen here is analogous to other well known orthogonal polynomial families' derivatives, for example the Jacobi polynomials on the interval, as seen in the DLMF \cite[(18.9.3)]{DLMF}, on the triangle \cite{olver2018recurrence}, and on the disk-slice \cite{snowball2019sparse}. The operators we define here are for partial derivatives with respect to the spherical coordinates $(\varphi, \theta)$, so that we can more easily apply the operators to PDEs on the surface of a sphere (for example, surface Laplacian operator in the Poisson equation). With the OP ordering by Fourier mode $k$ defined in equations (\ref{eqn:OPdefNka}, \ref{eqn:OPdefN0a}, \ref{eqn:OPdefNa}) these rotationally invariant operators are block-diagonal, meaning simple and parallelisable practical application.

\begin{theorem}\label{theorem:sparsityofdifferentialoperators}
	The operator matrices $D_\varphi^{(a)}, \: W_\varphi^{(a)}, \: D_\theta, \: \calL^{(a)\to(a+\tilde a)}, \: \calL_W^{(a)\to(a-\tilde a)}, \: \Delta^{(1)}_W$ from Definition \ref{def:differentialoperators} are sparse, with banded-block-banded structure. More specifically:
\begin{itemize}
	\item $D_\varphi^{(a)}$ is block-diagonal with sub-block-bandwidths $(2, 4)$
  	\item $W_\varphi^{(a)}$ is block-diagonal with sub-block-bandwidths $(4, 2)$
	\item $D_\theta$ is block-diagonal with sub-block-bandwidths $(1, 1)$
	\item $\calL^{(a)\to(a+\tilde a)}$ is block-diagonal with sub-block-bandwidths $(0, 4)$
	\item $\calL_W^{(a)\to(a-\tilde a)}$ is block-diagonal with sub-block-bandwidths $(4, 0)$
	\item $\Delta^{(1)}_W$ is block-diagonal with sub-block-bandwidths $(2, 2)$
\end{itemize}
\end{theorem}

In order to show the last part of Theorem \ref{theorem:sparsityofdifferentialoperators}, we require the following short lemma.

\begin{lemma}\label{lemma:Rsecondderivative}
	For any general parameter $a$ and any $n = 0,1,\dots$, $k = 0,\dots,n$ we have that
\begin{align*}
	&\ddx{z}[\genjacw^{(a+1, 2(k+1))} \: \genjacnmk^{(a,2k) \: \prime}] \\
	&\quad \quad = \genjacw^{(a+1, 2(k+1))} \: \genjacnmk^{(a,2k) \: \prime \prime} - 2(k+1)z \genjacw^{(a+1, 2k)} \: \genjacnmk^{(a,2k) \: \prime} + (a+1) \genjacw^{(a, 2(k+1))} \: \genjacnmk^{(a,2k) \: \prime} \\
	&\quad \quad = \sum_{m = n-1}^{n+1} \: c_{m,k} \: \genjacw^{(a, 2k)} \: \genjac_{m-k}^{(a,2k)}
\end{align*}
where 
\begin{align*}
	c_{m,k} = - \frac{1}{\normgenjac^{(a,2k)}} \: \int_\alpha^1 \: \genjacnmk^{(a,2k) \: \prime} \: \genjac_{m-k}^{(a,2k) \: \prime} \: \genjacw^{(a+1, 2(k+1))} \: \D z
\end{align*}
\end{lemma}

\begin{proof}[Proof of Lemma \ref{lemma:Rsecondderivative}]
	Since $\ddx{z}[\genjacw^{(a+1, 2(k+1))} \: \genjacnmk^{(a,2k) \: \prime}] = \genjacw^{(a, 2k)} \: r_{n-k+1}$ where $r_{n-k+1}$ is a degree $n - k + 1$ polynomial, we have that 
\begin{align*}
	\ddx{z}[\genjacw^{(a+1, 2(k+1))} \: \genjacnmk^{(a,2k) \: \prime}] =\sum_{m=0}^{n-k+1} \: \tilde c_{\{n,k\},m} \: \genjacw^{(a, 2k)} \: \genjac_{m}^{(a,2k)}
\end{align*}
for some coefficients $\tilde c_{\{n,k\},m}$. These coefficients are given by
\begin{align*}
	\tilde c_{\{n,k\},m} &= \frac{1}{\normgenjac^{(a,2k)}} \: \ip<\ddx{z}[\genjacw^{(a+1, 2(k+1))} \: \genjacnmk^{(a,2k) \: \prime}], \genjac_{m}^{(a,2k)}>_{\genjacw^{(0,0)}} \\
	&= - \frac{1}{\normgenjac^{(a,2k)}}  \: \int_\alpha^1 \: \genjacnmk^{(a,2k) \: \prime} \: \genjac_{m}^{(a,2k) \: \prime} \: \genjacw^{(a+1, 2(k+1))} \: \D z
\end{align*}
We show that these are zero for $m < n - k - 1$ by integrating twice by parts:
\begin{align*}
	&\ip<\ddx{z}[\genjacw^{(a+1, 2(k+1))} \: \genjacnmk^{(a,2k) \: \prime}], \genjac_{m}^{(a,2k)}>_{\genjacw^{(0,0)}} \\
	&\quad \quad \quad = - \int_\alpha^1 \: \genjacnmk^{(a,2k) \: \prime} \: \genjac_{m-k}^{(a,2k) \: \prime} \: \genjacw^{(a+1, 2(k+1))} \: \D z \\
	&\quad \quad \quad = \int_\alpha^1 \: \genjacnmk^{(a,2k) \: \prime} \: [(a+1) \genjac_{m}^{(a,2k) \: \prime} \: \genjacw^{(0, 2)} \\
	&\quad \quad \quad \quad \quad \quad \quad \quad \quad \quad - 2(k+1) z \genjac_{m}^{(a,2k) \: \prime} \: \genjacw^{(1, 0)} + \genjac_{m}^{(a,2k) \: \prime \prime} \: \genjacw^{(1, 2)}] \: \genjacw^{(a, 2k)} \: \D z
\end{align*}
which is indeed zero for $m < n - k - 1$ by orthogonality.
\end{proof}

\begin{proof}[Proof of Theorem \ref{theorem:sparsityofdifferentialoperators}]

For the operator $D_\theta$ for partial differentiation by $\theta$, we simply have that
\begin{align*}
	\pptheta \scopnkia(x,y,z) &= \genjacnmk^{(a,2k)}(z) \: \rho(z)^{k} \: \ddx{\theta} \chki(\theta) \\
	&= 
	\begin{cases}
		(-1)^{i+1} \: k \: \scopa_{n,k,|i-1|}(x,y,z) & k > 0 \\
		0 & k = 0
	\end{cases}.
\end{align*}

We now proceed with the case for the operator $D_\varphi^{(a)}$ for partial differentiation by $\varphi$. The entries of the operator are given by the coefficients in the expansion 
$$
\rhoppphi \scopnkia = \sum_{m=0}^{n+1} \sum_{j=0}^m \sum_{h=0}^1 c_{m,j,h} \: \scopmjh^{(a+1)},
$$ 
where the coefficients are 
\begin{align*}
	c_{m,j,h} = \norm{\scopmjh^{(a+1)}}^{-2}_{\scop^{(a+1)}} \ip<\rhoppphi \scopnkia, \: \scopmjh^{(a+1)}>_{\scop^{(a+1)}}.
\end{align*}
Now, note that:
\begin{align*}
	\genjacw^{(a,b) \: \prime}(z) &= a \: \genjacw^{(a-1,b)}(z) + c \: \rho(z) \: \rho'(z) \:\genjacw^{(a,b-2)}(z), \\
	\rho(z) \: \rho'(z) &= -z \\
	\ppphi \scopnkia(x,y,z) &= -\rho(z) \: \ddz \Big[ \rho(z)^k \: \genjacnmk^{(a,2k)}(z) \Big] \chki(\theta), \\
	\ppphi \Big[ \genjacw^{(a,0)}(z) \: \scopnkia(x,y,z) \Big] &= -\rho(z) \: \ddz \Big[ \genjacw^{(a,k)}(z) \: \genjacnmk^{(a,2k)}(z) \Big] \chki(\theta).
\end{align*}
Then, 
\begin{align*}
	&\ip<\rhoppphi \scopnkia, \: \scopmjh^{(a+1)}>_{\scop^{(a+1)}} \\
	&\quad = - \int_\alpha^1 \Big( \int_0^{2\pi} \: \rho(z)^2 \ddz \: [\genjacnmk^{(a,2k)}(z) \: \rho(z)^k] \: \genjacmmj^{(a+1,2j)}(z) \: \rho(z)^j \: \chki(\theta) \: \chjh(\theta) \: \D \theta \Big) \: \genjacw^{(a+1, 0)} \: \D z \\
	&\quad = \Big( \int_0^{2\pi} \: \chki(\theta) \: \chjh(\theta) \: \D \theta \Big) \: \Big( \int_\alpha^1 \: \genjacmmj^{(a+1,2j)} \: [k z \genjacnmk^{(a, 2k)} - \rho^2 \genjacnmk^{(a, 2k) \: \prime}] \: \genjacw^{(a+1, k + j)} \: \D z \Big) \\
	&\quad = \pi \: \delta_{k,j} \: \delta_{i,h} \: \int_\alpha^1 \: \genjac_{m-k}^{(a+1,2k)} \: [k z \genjacnmk^{(a, 2k)} - \rho^2 \genjacnmk^{(a, 2k) \: \prime}] \: \genjacw^{(a+1, 2k)} \: \D z \\
	&\quad = \pi \: \delta_{k,j} \: \delta_{i,h} \: \int_\alpha^1 \: \genjacnmk^{(a, 2k)} \: \Big\{  k z \: \genjac_{m-k}^{(a+1,2k)} \: \genjacw^{(1, 0)} + \genjac_{m-k}^{(a+1,2k) \: \prime} \: \genjacw^{(1, 2)}\\
	&\quad \quad \quad \quad \quad \quad \quad \quad \quad \quad \quad \quad \quad + a \: \rho^2 \: \genjac_{m-k}^{(a+1,2k)} - (2k+2) z \: \genjac_{m-k}^{(a+1,2k)} \: \genjacw^{(1, 0)} \Big\} \: \genjacw^{(a, 2k)} \: \D z
\end{align*}
which is zero for $j \ne k$, $h \ne i$, and $m < n - 2$ by orthogonality.

Similarly for the operator $W_\varphi^{(a)}$ for partial differentiation by $\varphi$ on the weighted space, the entries of the operator are given by the coefficients in the expansion $\rhoppphi (\genjacw^{(a,0)} \: \scopnkia) = \sum_{m=0}^{n+2} \sum_{j=0}^m \sum_{h=0}^1 c_{m,j,h} \: \genjacw^{(a-1,0)} \: \scopmjh^{(a-1)}$, where the coefficients are
\begin{align*}
	c_{m,j,h} = \norm{\scopmjh^{(a-1)}}^{-2}_{\scop^{(a-1)}} \ip<\rhoppphi (\genjacw^{(a,0)} \: \scopnkia), \: \scopmjh^{(a-1)}>_{\scop^{(0)}}.
\end{align*}
Now,
\begin{align*}
	&\ip<\rhoppphi (\genjacw^{(a,0)} \: \scopnkia), \: \scopmjh^{(a-1)}>_{\scop^{(0)}} \\
	&\quad = - \int_\alpha^1 \Big( \int_0^{2\pi} \: \rho(z)^2 \ddz \: [\genjacnmk^{(a,2k)}(z) \: \genjacw^{(a, k)}(z)] \: \genjacmmj^{(a-1,2j)}(z) \: \rho(z)^j \: \chki(\theta) \: \chjh(\theta) \: \D \theta \Big) \: \D z \\
	&\quad = \Big( \int_0^{2\pi} \: \chki(\theta) \: \chjh(\theta) \: \D \theta \Big) \\
	&\quad \quad \quad \quad \cdot \: \Big( \int_\alpha^1 \: \genjacmmj^{(a-1,2j)} \: [k z \genjacnmk^{(a, 2k)} \: \genjacw^{(1, 0)} - \genjacnmk^{(a, 2k) \: \prime} \: \genjacw^{(1, 2)} - a \: \genjacnmk^{(a, 2k)} \: \rho^2] \: \genjacw^{(a-1, k + j)} \: \D z \Big) \\
	&\quad = \pi \: \delta_{k,j} \: \delta_{i,h} \:  \int_\alpha^1 \: \genjac_{m-k}^{(a-1,2k)} \: [k z \genjacnmk^{(a, 2k)} \: \genjacw^{(1, 0)} - \genjacnmk^{(a, 2k) \: \prime} \: \genjacw^{(1, 2)} - a \: \genjacnmk^{(a, 2k)} \: \rho^2] \: \genjacw^{(a-1, 2k)} \: \D z \\
	&\quad = \pi \: \delta_{k,j} \: \delta_{i,h} \: \int_\alpha^1 \: \genjacnmk^{(a, 2k)} \: \Big\{  k z \: \genjac_{m-k}^{(a-1,2k)} \: \genjacw^{(1, 0)} - a \: \rho^2 \: \genjac_{m-k}^{(a-1,2k)} + \genjac_{m-k}^{(a-1,2k) \: \prime} \: \genjacw^{(1, 2)} \\
	&\quad \quad \quad \quad \quad \quad \quad \quad \quad \quad \quad \quad \quad + a \: \rho^2 \: \genjac_{m-k}^{(a-1,2k)} - (2k+2) z \: \genjac_{m-k}^{(a-1,2k)} \: \genjacw^{(1, 0)} \Big\} \: \genjacw^{(a-1, 2k)} \: \D z \\
	&\quad = \pi \: \delta_{k,j} \: \delta_{i,h} \: \int_\alpha^1 \: \genjacnmk^{(a, 2k)} \: \Big\{  k z \: \genjac_{m-k}^{(a-1,2k)} + \genjac_{m-k}^{(a-1,2k) \: \prime} \: \rho^2- (2k+2) z \: \genjac_{m-k}^{(a-1,2k)} \Big\} \: \genjacw^{(a, 2k)} \: \D z
\end{align*}
which is zero for $j \ne k$, $h \ne i$, and $m < n - 1$ by orthogonality.

We move on to the spherical Laplacian operators. Note that the Laplacian acting on the weighted and non-weighted spherical cap OP $\scopnki^{(a)}$ yield
\begin{align}
	\DeltaS \: \scopnki^{(a)} &= {1 \over \rho} \ppx{\varphi} \Big( \rho \ppx{\varphi} [\genjacnmk^{(a,2k)}(\cosphi) \: \sin^k\varphi] \Big) \chki(\theta) \nonumber \\
	&\quad \quad \quad \quad + \genjacnmk^{(a,2k)}(\cosphi) \: \sin^{k-2}\varphi \pmpxm{\theta}{2} \chki(\theta) \nonumber \\
	&= \chki(\theta) \rho(z)^k \Big\{ -k(k+1) \genjacnmk^{(a,2k)}(z) - 2(k+1) z \: \genjacnmk^{(a,2k) \: \prime}(z) \nonumber \\
	&\quad \quad \quad \quad \quad \quad \quad + \rho(z)^2 \genjacnmk^{(a,2k) \: \prime \prime}(z) \Big\}, \label{eqn:oplaplacian} \\
	\DeltaS \big(\genjacw^{(a,0)} \: \scopnki^{(a)} \big) &= {1 \over \rho} \ppx{\varphi} \Big( \rho \ppx{\varphi} [\genjacw^{(a,0)}(\cosphi) \: \genjacnmk^{(a,2k)}(\cosphi) \: \sin^k\varphi] \Big) \chki(\theta) \nonumber \\
	&\quad \quad \quad \quad + \genjacw^{(a,0)}(\cosphi) \: \genjacnmk^{(a,2k)}(\cosphi) \: \sin^{k-2}\varphi \pmpxm{\theta}{2} \chki(\theta) \nonumber \\
	&= \chki(\theta) \Big\{\genjacnmk^{(a,2k)}(z) \big[ -k(k+1) \genjacw^{(a,k)}(z) - 2a(k+1) z \: \genjacw^{(a-1,k)}(z) \big] \nonumber \\
	&\quad \quad \quad \quad \quad + a(a-1) \genjacnmk^{(a,2k)}(z) \: \genjacw^{(a-2,k+1)}(z) \nonumber \\
	&\quad \quad \quad \quad \quad + \genjacnmk^{(a,2k) \: \prime}(z) \big[ -2(k+1) z \: \genjacw^{(a,k)}(z) + 2a \: \genjacw^{(a-1,k+2)}(z) \big] \nonumber \\
	&\quad \quad \quad \quad \quad + \genjacnmk^{(a,2k) \: \prime \prime}(z) \: \genjacw^{(a,k+2)}(z) \Big\}. \label{eqn:woplaplacian}
\end{align}

For the operator $\calL^{(a)\to(a+\tilde a)}$ for the surface Laplacian on a non-weighted space, the entries of the operator are given by the coefficients in the expansion 
$$\DeltaS \scopnki^{(a)} = \sum_{m=0}^{n} \sum_{j=0}^m \sum_{h=0}^1 c_{m,j,h} \: \scopmjh^{(a + \tilde a)},$$
 where the coefficients are
\begin{align*}
	c_{m,j,h} = \norm{\scopmjh^{(a + \tilde a)}}^{-2}_{\scop^{(a + \tilde a)}} \ip<\DeltaS \scopnki^{(a)}, \: \scopmjh^{(a + \tilde a)}>_{\scop^{(a + \tilde a)}}.
\end{align*}
Using \bsrefeqn{eqn:oplaplacian}, and integrating by parts twice, we then have that
\begin{align*}
	&\ip<\DeltaS \scopnki^{(a)}, \: \scopmjh^{(a + \tilde a)}>_{\scop^{(a + \tilde a)}} \\
	&\quad = \Big( \int_0^{2\pi} \: \chki(\theta) \: \chjh(\theta) \: \D \theta \Big) \\
	&\quad \quad \quad \quad \cdot \: \Big( \int_\alpha^1 \: \genjacmmj^{(a + \tilde a, 2j)} \: \genjacw^{(a + \tilde a), k+j)} \: \Big\{-k(k+1) \genjacnmk^{(a,2k)} - 2(k+1) z \: \genjacnmk^{(a,2k) \: \prime} \\
	&\quad \quad \quad \quad \quad \quad \quad \quad \quad \quad \quad \quad \quad \quad \quad \quad \quad + \rho(z)^2 \genjacnmk^{(a,2k) \: \prime \prime} \Big\} \: \D z \Big) \\
	&\quad = \pi \: \delta_{k,j} \: \delta_{i,h} \: \int_\alpha^1 \: \genjac_{m-k}^{(a+\tilde a,2k)} \: \genjacw^{(a+\tilde a, 0)} \Big( -k(k+1) \genjacnmk^{(a,2k)} \: \rho^{2k} + \ddx{z} [\genjacnmk^{(a,2k) \: \prime} \: \rho^{2(k+1)}] \Big) \: \D z \\
	&\quad = \pi \: \delta_{k,j} \: \delta_{i,h} \: \int_\alpha^1 \: \Big\{ -k(k+1) \genjac_{m-k}^{(a+\tilde a,2k)} \: \genjacnmk^{(a,2k)} \: \genjacw^{(a+\tilde a, 2k)} \\
	&\quad \quad \quad \quad \quad \quad \quad \quad \quad \quad - \genjacnmk^{(a,2k) \: \prime} \genjacw^{(a+\tilde a-1, 2k)} \: [\genjac_{m-k}^{(a+\tilde a,2k) \: \prime} \: \genjacw^{(1,0)} + (a + \tilde a) \genjac_{m-k}^{(a+\tilde a,2k)}] \Big\} \: \D z \\
	&\quad = \pi \: \delta_{k,j} \: \delta_{i,h} \: \int_\alpha^1 \: \genjacnmk^{(a,2k)} \: \genjacw^{(a, 2k)} \: r_{m-k+\tilde a} \: \D z
\end{align*}
where $r_{m-k+\tilde a}$ is a degree $m - k + \tilde a$ polynomial in $z$, and so the above is zero for $n  - k > m - k + \tilde a \iff m < n - \tilde a$.

For the operator $\calL_W^{(a)\to(a-\tilde a)}$ for the surface Laplacian on a weighted space, the entries of the operator are given by the coefficients in the expansion 
\begin{align*}
	\DeltaS \big(\genjacw^{(a,0)} \: \scopnki^{(a)} \big) = \sum_{m=0}^{n} \sum_{j=0}^m \sum_{h=0}^1 c_{m,j,h} \: \genjacw^{(a-\tilde a,0)} \: \scopmjh^{(a - \tilde a)}, 
\end{align*}	
where the coefficients are
\begin{align*}
	c_{m,j,h} = \norm{\scopmjh^{(a - \tilde a)}}^{-2}_{\scop^{(a - \tilde a)}} \ip<\DeltaS \big(\genjacw^{(a,0)} \: \scopnki^{(a)} \big), \: \scopmjh^{(a - \tilde a)}>_{\scop^{(0)}}.
\end{align*}
Using \bsrefeqn{eqn:woplaplacian}, and integrating by parts thrice, we then have that
\begin{align*}
	&\ip<\DeltaS \big(\genjacw^{(a,0)} \: \scopnki^{(a)} \big), \: \scopmjh^{(a - \tilde a)}>_{\scop^{(0)}} \\
	&\quad = \Big( \int_0^{2\pi} \: \chki(\theta) \: \chjh(\theta) \: \D \theta \Big) \\
	&\quad \quad  \cdot \: \Big( \int_\alpha^1 \: \genjacmmj^{(a - \tilde a, 2j)} \: \genjacw^{(a - 2, k+j)} \: \Big\{\genjacnmk^{(a,2k)} [-k(k+1) \genjacw^{(2,0)} - 2a(k+1)z \: \genjacw^{(1,0)} + a (a-1) \rho^2] \\
	&\quad \quad \quad \quad \quad \quad \quad \quad \quad \quad \quad \quad \quad \quad \quad \quad \quad + \genjacnmk^{(a,2k) \: \prime} [-2(k+1) z \: \genjacw^{(2,0)} + 2a \genjacw^{(1,2)}] \\
	&\quad \quad \quad \quad \quad \quad \quad \quad \quad \quad \quad \quad \quad \quad \quad \quad \quad + \genjacnmk^{(a,2k) \: \prime \prime} \: \genjacw^{(2,2)} \Big\} \: \D z \Big) \\
	&\quad = \pi \: \delta_{k,j} \: \delta_{i,h} \: \int_\alpha^1 \: \Big\{ \genjac_{m-k}^{(a - \tilde a, 2k)} \: \genjacnmk^{(a,2k)} \: \genjacw^{(a - 2, 2k)} \: [-k(k+1) \genjacw^{(2,0)} - 2a(k+1)z \: \genjacw^{(1,0)} + a (a-1) \rho^2] \\
	&\quad \quad \quad \quad \quad \quad \quad \quad \quad \quad + a \: \genjacnmk^{(a,2k) \: \prime} \: \genjac_{m-k}^{(a - \tilde a, 2k)} \: \genjacw^{(a - 1, 2k+2)} \\
	&\quad \quad \quad \quad \quad \quad \quad \quad \quad \quad + \genjac_{m-k}^{(a - \tilde a, 2k)} \: \ddx{z} [\genjacnmk^{(a,2k) \: \prime} \: \genjacw^{(a, 2k+2)}] \Big\} \: \D z \\
	&\quad = \pi \: \delta_{k,j} \: \delta_{i,h} \: \int_\alpha^1 \: \Big\{ \genjac_{m-k}^{(a - \tilde a, 2k)} \: \genjacnmk^{(a,2k)} \: \genjacw^{(a - 2, 2k)} \: [-k(k+1) \genjacw^{(2,0)} - 2a(k+1)z \: \genjacw^{(1,0)} + a (a-1) \rho^2] \\
	&\quad \quad \quad \quad \quad \quad \quad + a \: \genjacnmk^{(a,2k) \: \prime} \: \genjac_{m-k}^{(a - \tilde a, 2k)} \: \genjacw^{(a - 1, 2k+2)} \\
	&\quad \quad \quad \quad \quad \quad \quad + \genjacnmk^{(a,2k)} \: \genjacw^{(a-1, 2k)} [\genjac_{m-k}^{(a - \tilde a, 2k) \: \prime \prime} \: \genjacw^{(1, 2)} + a \genjac_{m-k}^{(a - \tilde a, 2k) \: \prime} \: \rho^2 - 2(k+1) z \: \genjac_{m-k}^{(a - \tilde a, 2k)} \: \genjacw^{(1, 0)}] \Big\} \: \D z \\
	&\quad = \pi \: \delta_{k,j} \: \delta_{i,h} \: \int_\alpha^1 \: \Big\{ \genjac_{m-k}^{(a - \tilde a, 2k)} \: \genjacnmk^{(a,2k)} \: \genjacw^{(a - 2, 2k)} \: [-k(k+1) \genjacw^{(2,0)} - 2a(k+1)z \: \genjacw^{(1,0)} + a (a-1) \rho^2] \\
	&\quad \quad \quad \quad \quad \quad \quad \quad \quad \quad + \genjacnmk^{(a,2k)} \: \genjacw^{(a - 1, 2k+2)} [\genjac_{m-k}^{(a - \tilde a, 2k) \: \prime \prime} \: \rho^2 - 2(k+1) z \: \genjac_{m-k}^{(a - \tilde a, 2k) \: \prime} ] \\
	&\quad \quad \quad \quad \quad \quad \quad \quad \quad \quad + a [\genjacnmk^{(a,2k)} \: \genjac_{m-k}^{(a - \tilde a, 2k) \: \prime} + \genjac_{m-k}^{(a - \tilde a, 2k)} \: \genjacnmk^{(a, 2k) \: \prime}] \: \genjacw^{(a-1, 2k+2)} \Big\} \: \D z \\
	&\quad = \pi \: \delta_{k,j} \: \delta_{i,h} \: \int_\alpha^1 \: \Big\{ \genjac_{m-k}^{(a - \tilde a, 2k)} \: \genjacnmk^{(a,2k)} \: \genjacw^{(a - 2, 2k)} \: [-k(k+1) \genjacw^{(2,0)} - 2a(k+1)z \: \genjacw^{(1,0)} + a (a-1) \rho^2] \\
	&\quad \quad \quad \quad \quad \quad \quad \quad \quad \quad + \genjacnmk^{(a,2k)} \: \genjacw^{(a - 1, 2k+2)} [\genjac_{m-k}^{(a - \tilde a, 2k) \: \prime \prime} \: \rho^2 - 2(k+1) z \: \genjac_{m-k}^{(a - \tilde a, 2k) \: \prime} ] \\
	&\quad \quad \quad \quad \quad \quad \quad \quad \quad \quad - a \genjacnmk^{(a,2k)} \: \genjac_{m-k}^{(a - \tilde a, 2k)} \: \genjacw^{(a-2, 2k)} \: [(a-1) \rho^2 - 2(k+1) z \: \genjacw^{(1, 0)}] \Big\} \: \D z \\
	&\quad = \pi \: \delta_{k,j} \: \delta_{i,h} \: \int_\alpha^1 \: \genjacnmk^{(a,2k)} \: \genjacw^{(a, 2k)} \: r_{m-k} \: \D z
\end{align*}
where $r_{m-k}$ is a degree $m - k$ polynomial in $z$, and so the above is zero for $n - k > m - k \iff m < n$.

Finally, fix $a = 1$. For the operator $\Delta^{(1)}_W$ for the Laplacian on the weighted space, the entries of the operator are given by the coefficients in the expansion $\DeltaS \big(\genjacw^{(1,0)} \: \scopnki^{(1)} \big) = \sum_{m=0}^{n+2} \sum_{j=0}^m \sum_{h=0}^1 c_{m,j,h} \: \scopmjh^{(1)}$, where the coefficients are given by
\begin{align*}
	c_{m,j,h} = \norm{\scopmjh^{(1)}}^{-2}_{\scop^{(1)}} \ip<\DeltaS \big(\genjacw^{(1,0)} \: \scopnki^{(1)} \big), \: \scopmjh^{(1)}>_{\scop^{(1)}}.
\end{align*}
Using \bsrefeqn{eqn:woplaplacian} with $a = 1$, and \bsreflemma{lemma:Rsecondderivative}, we then have that
\begin{align*}
	&\ip<\DeltaS \big(\genjacw^{(1,0)} \: \scopnki^{(1)} \big), \: \scopmjh^{(1)}>_{\scop^{(1)}} \\
	&\quad = \Big( \int_0^{2\pi} \: \chki(\theta) \: \chjh(\theta) \: \D \theta \Big) \\
	&\quad \quad \quad \quad \cdot \: \Big( \int_\alpha^1 \: \genjacmmj^{(1,2j)} \: \Big\{ \genjacnmk^{(1, 2k)} \: [- k^2 \genjacw^{(1, k)} - \genjacw^{(1, k)} - 2(k + 1)z \genjacw^{(0, k)}] \\
	&\quad \quad \quad \quad \quad \quad \quad \quad \quad \quad \quad \quad + \genjacnmk^{(1, 2k) \: \prime} \: [2 \genjacw^{(0, k+2)} - 2(k + 1)z \genjacw^{(1, k)}] \\
	&\quad \quad \quad \quad \quad \quad \quad \quad \quad \quad \quad \quad + \genjacnmk^{(1, 2k) \: \prime \prime} \: \genjacw^{(1, k+2)} \Big\} \: \genjacw^{(1, j)} \: \D z \Big) \\
	&\quad = \pi \: \delta_{k,j} \: \delta_{i,h} \: \int_\alpha^1 \: \genjac_{m-k}^{(1,2k)} \: \Big\{ \genjacnmk^{(1, 2k)} [-k(k+1) \genjacw^{(1, 0)} - 2(k+1)z + c_{n,k}] \\
	&\quad \quad \quad \quad \quad \quad \quad \quad \quad \quad \quad \quad + c_{n-1,k} \genjac_{n-k-1}^{(1, 2k)} + c_{n+1,k} \genjac_{n-k+1}^{(1, 2k)} \Big\} \: \genjacw^{(1, 2k)} \: \D z \\
	&\quad = - \pi \: \delta_{k,j} \: \delta_{i,h} (\delta_{m,n-1} + \delta_{m,n} + \delta_{m,n+1}) \: \int_\alpha^1 \: \Big\{ \genjacnmk^{(1,2k)} \: \genjac_{m-k}^{(1, 2k)} (k(k+1) \genjacw^{(1, 0)} + 2(k+1)z) \\ 
	&\quad \quad \quad \quad \quad \quad \quad \quad \quad \quad \quad \quad \quad \quad \quad \quad \quad \quad \quad \quad \quad + \genjacnmk^{(1,2k) \: \prime} \: \genjac_{m-k}^{(1,2k) \: \prime} \: \genjacw^{(2, 2(k+1))} \Big\} \: \D z
\end{align*}
where the $c_{n-1,k}, c_{n,k}, c_{n+1,k}$ are those derived in \bsreflemma{lemma:Rsecondderivative}.
\end{proof}

By applying these differential operators, we are (in some cases) incrementing or decrementing the parameter value $a$. It is therefore necessary to also be able to raise or lower the parameter by way of an independent operator. There exist conversion matrix operators that do exactly this, transforming the OPs from one (weighted or non-weighted) parameter space to another.
\begin{definition}\label{def:parametertransformationoperators}
Define the operator matrices $T^{(a)\to(a+\tilde a)}, \quad T_W^{(a)\to(a-\tilde a)}$ for conversion between non-weighted spaces and weighted spaces respectively according to
\begin{align*}
	\bigscopNa(x,y,z) &= \Big(T^{(a)\to(a+\tilde a)} \Big)^\top \: \bigscopN^{(a+\tilde a)}(x,y,z) \\
	\bigWNa(x,y,z) &= \Big(T_W^{(a)\to(a-\tilde a)} \Big)^\top \: \bigW_N^{(a-\tilde a)}(x,y,z)
\end{align*}
\end{definition}

\begin{lemma}\label{lemma:sparsityofparametertransformationoperators}
The operator matrices in Definition \ref{def:parametertransformationoperators} are sparse, with banded-block-banded structure. More specifically:
\begin{itemize}
	\item $T^{(a)\to(a+\tilde a)}$ is block-diagonal with sub-block bandwidths $(0,2\tilde a)$
	\item $T_W^{(a)\to(a-\tilde a)}$ is block-diagonal with sub-block bandwidths $(2\tilde a, 0)$
\end{itemize}
\end{lemma}

\begin{proof}
We proceed with the case for the non-weighted operators $T^{(a)\to(a+\tilde a)}$. Since $\{\scopmjh^{(a+\tilde{a})}\}$ for $m = 0,\dots,n$, $j = 0,\dots,m$, $h = 0,1$ is an orthogonal basis for any degree $n$ polynomial, we can expand $\scopnkia = \sum_{m=0}^{n} \sum_{j=0}^m t_{m,j} \: \scopmjh^{(a+\tilde{a})}$. The coefficients of the expansion are then the entries of the operator matrix. We will show that the only non-zero coefficients are for $k = j$, $i = h$ and $m \ge  n - \tilde a$. Note that
\begin{align*}
	t_{m,j} = \norm{\scopmjh^{(a+\tilde{a})}}^{-2}_{\scop^{(a+\tilde a)}} \: \ip< \scopnkia, \scopmjh^{(a+\tilde{a})}>_{\scop^{(a+\tilde a)}}.
\end{align*}
where
\begin{align*}
	\ip< \scopnkia, \scopmjh^{(a+\tilde{a})}>_{\scop^{(a+\tilde a)}} &= \Big( \int_0^{2\pi} \: \chki(\theta) \: \chjh(\theta) \: \D \theta \Big) \: \cdot \: \Big( \int_\alpha^1 \: \genjacnmk^{(a, 2k)} \: \genjacmmj^{(a+\tilde a,2j)} \: \rho^{k+j} \: \genjacw^{(a + \tilde a, 0)} \: \D z \Big) \\
	&= \pi \: \delta_{k,j} \: \delta_{i,h} \: \int_\alpha^1 \: \genjacnmk^{(a, 2k)} \: \genjac_{m-k}^{(a+\tilde a,2k)} \: \genjacw^{(a + \tilde a, 2k)} \: \D z
\end{align*}
which is zero for $n > m + \tilde a \iff m < n - \tilde a$. The sparsity argument for the weighted parameter transformation operator follows similarly.
\end{proof}

\begin{figure}[t]
	\centering 
	\begin{subfigure}{0.45\textwidth}
		\includegraphics[scale=0.3]{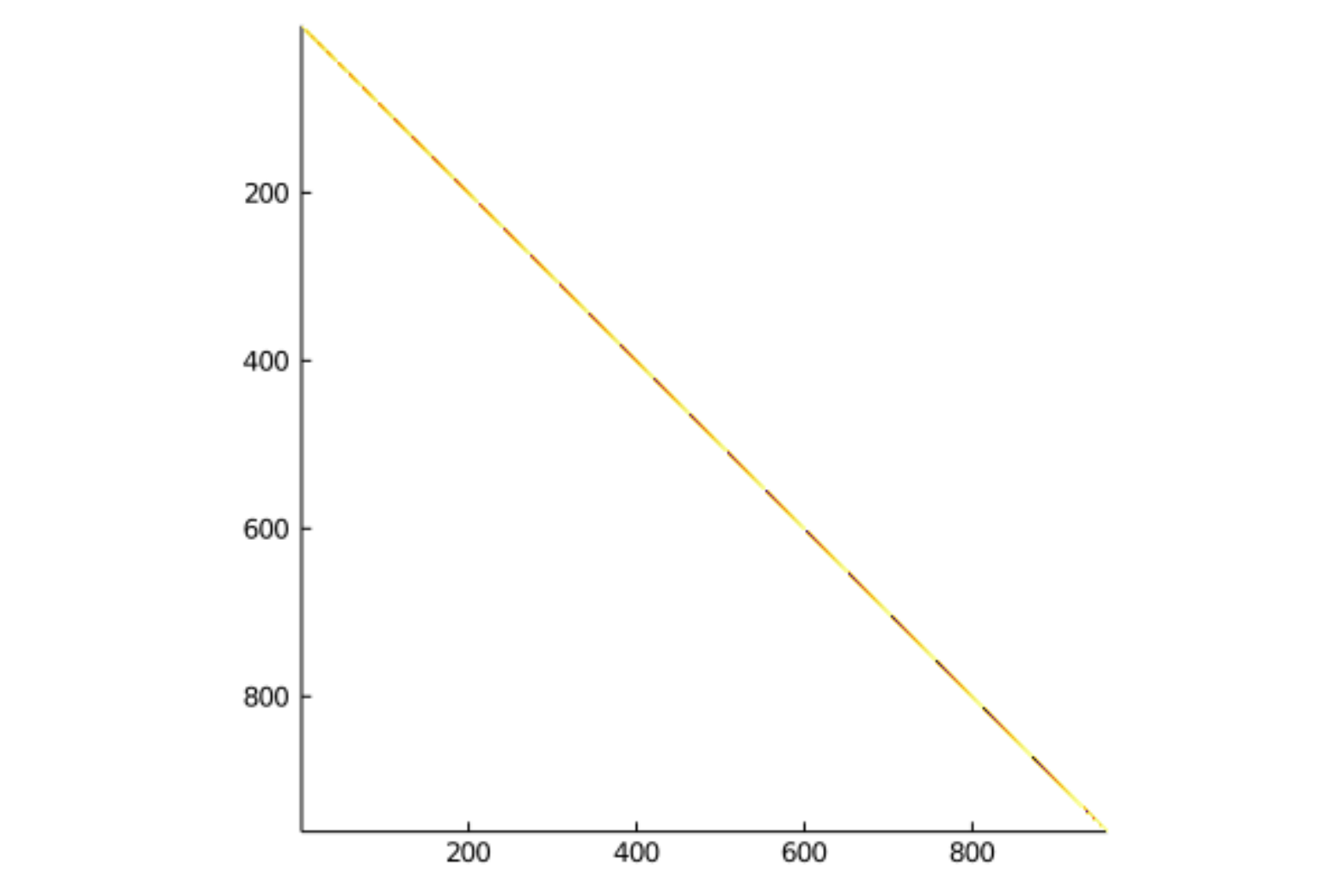}
		\centering
		\caption{The Laplace-Beltrami operator $\Delta^{(1)}_W$}
	\end{subfigure}\hfil 
	\begin{subfigure}{0.45\textwidth}
		\includegraphics[scale=0.3]{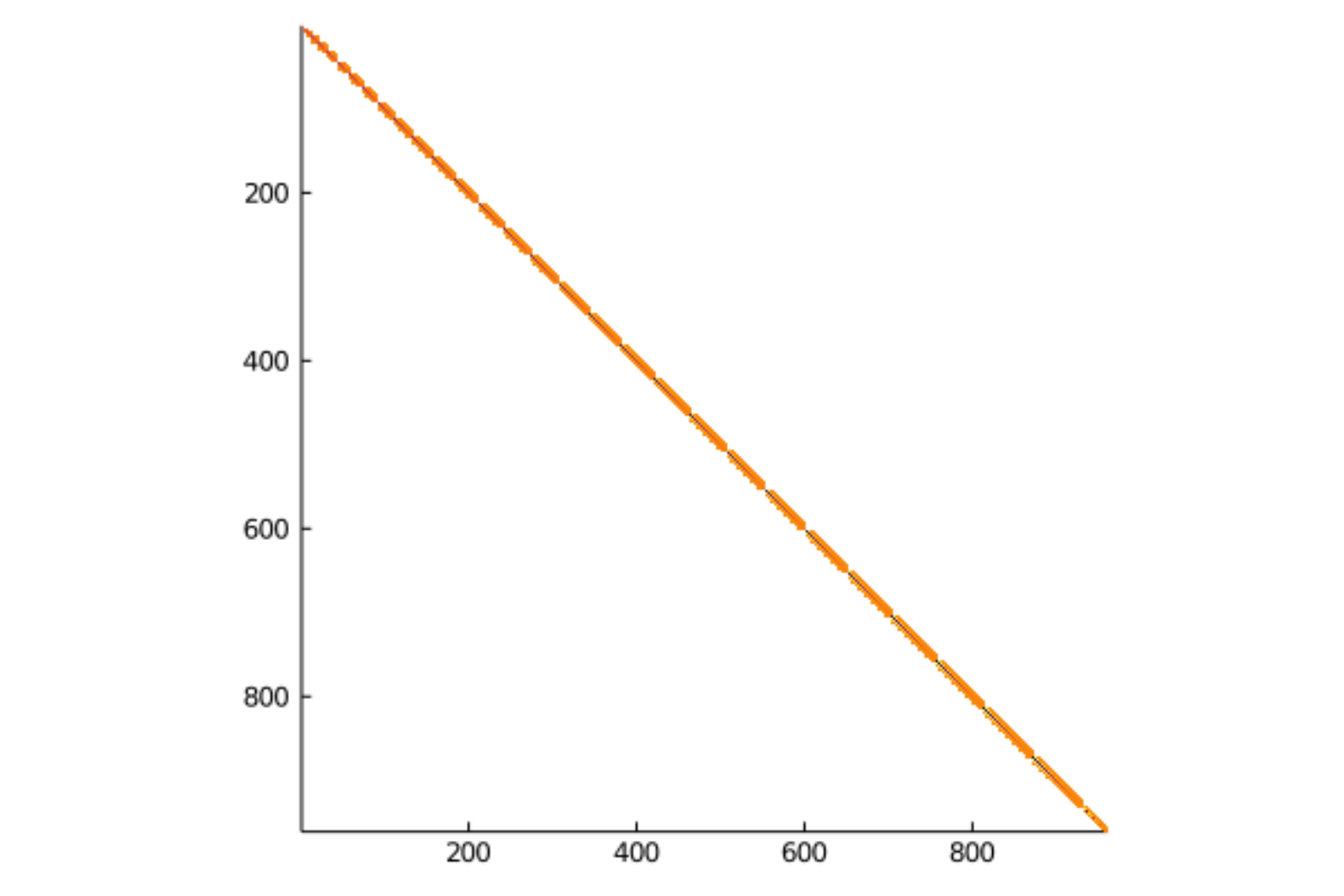}
		\centering
		\caption{The $\rho^2$-factored Laplace-Beltrami operator $D_\varphi^{(0)} \: W_\varphi^{(1)} + T^{(0)\to(1)} \: T_W^{(1)\to(0)} \: (D_\theta)^2$}
	\end{subfigure}\hfil 

	\medskip
	\begin{subfigure}{0.45\textwidth}
		\includegraphics[scale=0.3]{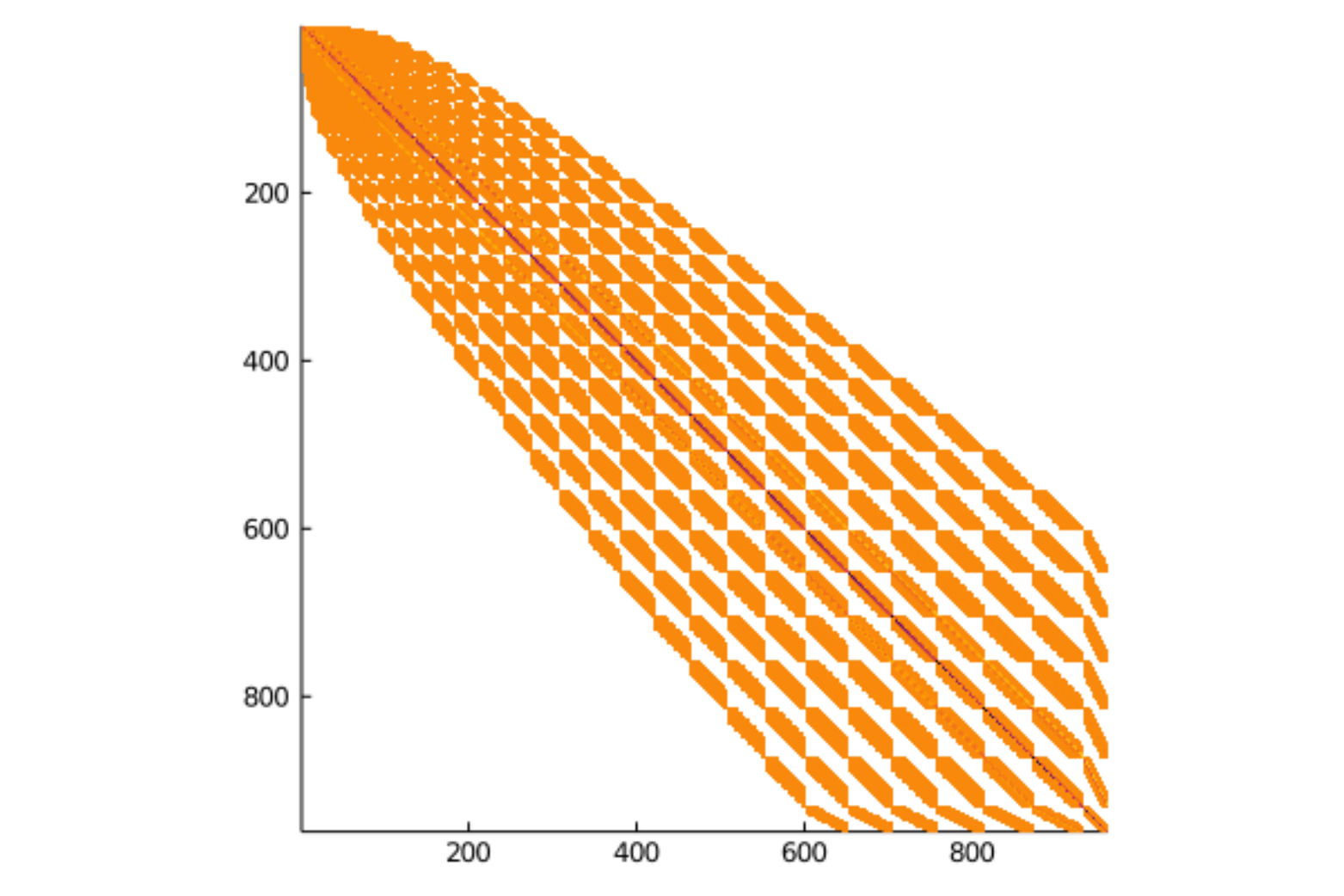}
		\centering
		\caption{A variable coefficient Helmholtz operator}
	\end{subfigure}\hfil 
	\begin{subfigure}{0.45\textwidth}
		\includegraphics[scale=0.3]{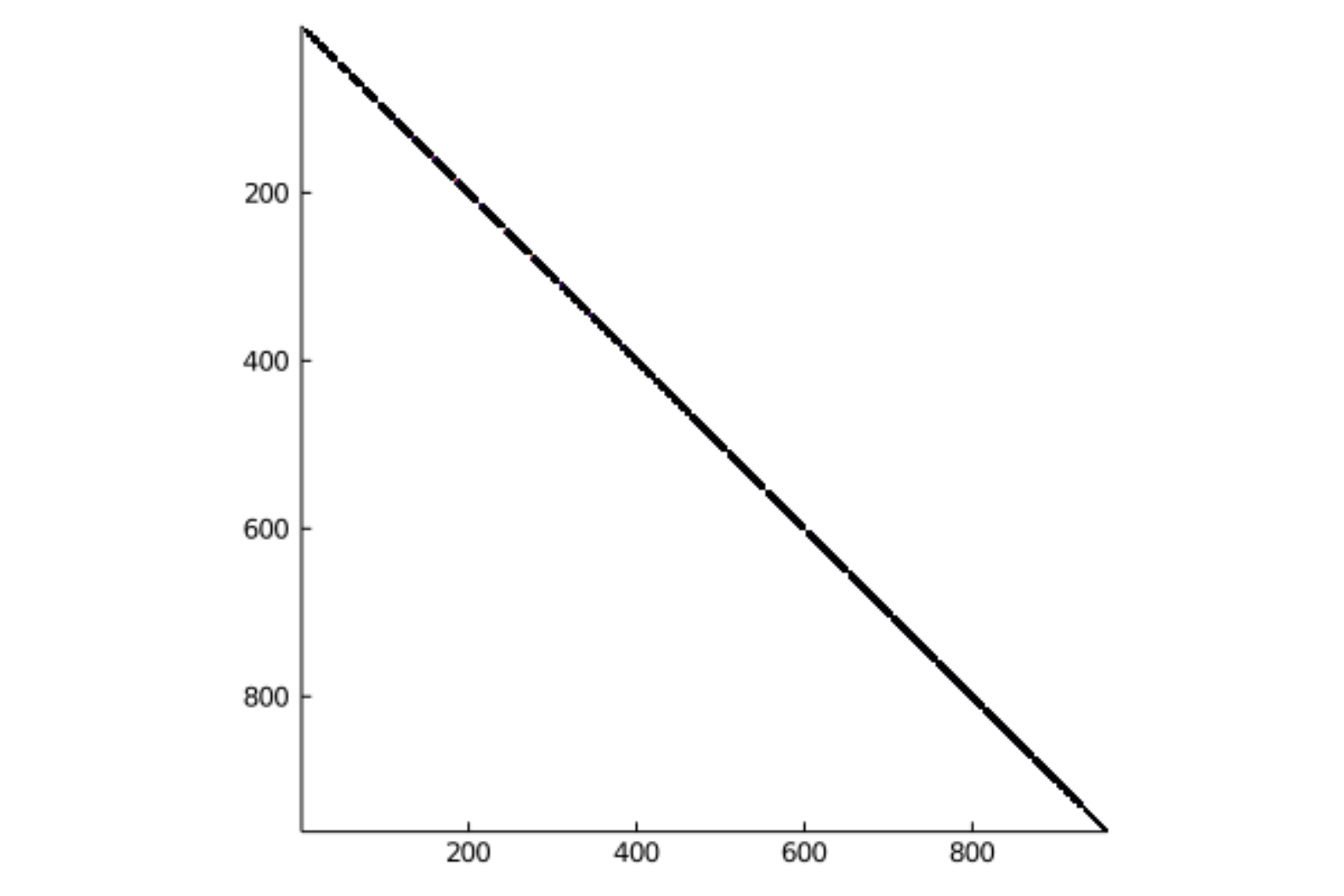}
		\centering
		\caption{The Biharmonic operator $\calB_W^{(2)}$} 
	\end{subfigure}\hfil 
	\caption{\enquote{Spy} plots of (differential) operator matrices, showing their sparsity. For (c), the the weighted variable coefficient Helmholtz operator is $\Delta^{(1)}_W + k^2 \: T^{(0)\to(1)} \: V({J_x^{(0)}}^\top, {J_y^{(0)}}^\top, {J_z^{(0)}}^\top) \: T_W^{(1)\to(0)}$ for $v(x,y,z) = 1 - (3(x-x_0)^2 + 5(y-y_0)^2 + 2(z-z_0)^2)$ where $(x_0, z_0) := (0.7, 0.2)$, $y_0 := \sqrt{1 - x_0^2 - z_0^2}$ and $k = 200$.}
	\label{fig:sparsity}
\end{figure}

\subsection{Further partial differential operators}\label{subsection:furtherdiffoperators}
General linear partial differential operators with polynomial variable coefficients can be constructed by composing the sparse representations for partial derivatives, conversion between bases, and Jacobi operators. As a canonical example, we can obtain the matrix operator for the $\rho^2$-factored spherical Laplacian $\rho(z)^2 \: \DeltaS$, that will take us from coefficients for expansion in the weighted space $\bigW_N^{(1)}(x,y,z) = \genjacw^{(1,0)}(z) \: \bigscopN^{(1)}(x,y,z)$ to coefficients in the non-weighted space $\bigscopN^{(1)}(x,y,z)$. Note that this construction will ensure the imposition of the Dirichlet zero boundary conditions on $\Omega$, similar to how the Dirichlet zero boundary conditions would be imposed for the operator $\Delta^{(1)}_W$ in Definition \ref{def:differentialoperators}. The matrix operator for this $\rho^2$-factored spherical Laplacian acting on the coefficients vector is then given by
\begin{align*}
	D_\varphi^{(0)} \: W_\varphi^{(1)} + T^{(0)\to(1)} \: T_W^{(1)\to(0)} \: (D_\theta)^2.
\end{align*}
Importantly, this operator will have banded-block-banded structure, and hence will be sparse, as seen in Figure \ref{fig:sparsity}.

Another desirable operator is the Biharmonic operator $\DeltaS^2$, for which we assume zero Dirichlet and Neumann conditions. That is, 
\begin{align*}
	u(x,y,z) &= 0, \quad \pfpx{n}{u}(x,y,z) = \nabla_S u(x,y,z) \cdot \unitvec{n}(x,y,z) = 0 \quad \text{for } (x,y,z) \in \partial\Omega
\end{align*}
where $\partial\Omega$ is the $z=\alpha$ boundary, and $\unitvec{n}(x,y,z)$ is the outward unit normal vector at the point $(x,y,z)$ on the boundary, i.e. $\unitvec{n}(x,y,z) = \unitvec{n}(\vec x) := {\vec x \over \norm{\vec x}} = \vec x$. The matrix operator for the Biharmonic operator will take us from coefficients in the space $\bigW^{(2)}(x,y,z)$ to coefficients in the space $\bigscopN^{(2)}(x,y,z)$. To construct this, we can simply multiply together two of the spherical Laplacian operators defined in \bsrefdef{def:differentialoperators}, namely $\calL^{(0)\to(2)}$ and $\calL_W^{(2)\to(0)}$:
\begin{align*}
	\calB^{(2)}_W := \calL^{(0)\to(2)} \: \calL_W^{(2)\to(0)}.
\end{align*}
Since the operator $\calL_W^{(2)\to(0)}$ acts on coefficients in the $\bigW^{(2)}(x,y,z)$ space, we ensure that we satisfy the zero Dirichlet and Neumann boundary conditions -- such a function could be written $u(x,y,z) = \genjacw^{(2,0)}(z) \: \tilde u(x,y,z)$ and thus its spherical gradient would be zero on the boundary $z = \alpha$. This allows us to apply the $\calL^{(0)\to(2)}$ operator after, safe in the knowledge that boundary conditions have been accounted for. The sparsity and structure of this biharmonic operator are seen in Figure \ref{fig:sparsity}.

\section{Computational aspects}\label{Section:Computation}

In this section we discuss how to expand and evaluate functions in our proposed basis, and take advantage of the sparsity structure in partial differential operators in practical computational applications.

\subsection{Constructing $\genjac_n^{(a,b)}(x)$}

It is possible to recursively obtain the recurrence coefficients for the $\{\genjac_n^{(a,b)}\}$ OPs in (\ref{eqn:Rrecurrence}), see \cite{snowball2019sparse}, by careful application of the Christoffel--Darboux formula \cite[18.2.12]{DLMF}.

\subsection{Quadrature rule on the spherical cap}\label{subsection:quadrule}

In this section we construct a quadrature rule exact for polynomials on the spherical cap $\Omega$ that can be used to expand functions in the OPs $\scopnkia(x,y,z)$ for a given parameter $a$.

\begin{theorem}\label{Theorem:quadrule}
Let $M_1, M_2 \in \N$ and denote the $M_1$ Gauss quadrature nodes and weights on $[\alpha,1]$ with weight $(t - \alpha)^a$ as $(t_j, w_j^{(t)})$. Further, denote the $M_2$ Gauss quadrature nodes and weights $[-1,1]$ with weight $(1 - x^2)^{-\half}$ as $(s_j, w_j^{(s)})$.
Define for $j = 1,\dots,M_1, \: l=1,\dots,M_2$:
\begin{align*}
	\big(x_{l+(j-1)M_2}, \: y_{l+(j-1)M_2} \big) &:= \rho(t_j) \: \mathbf{s}_l, \\
	z_{l+(j-1)M_2} &:= t_j, \\
	w_{l+(j-1)M_2} &:= w_j^{(t)} w_l^{(s)}.
\end{align*}
Let $f(x,y,z)$ be a function on $\Omega$, and $N \in \N$. The quadrature rule is then
\begin{align*}
	\int_\Omega f(x,y,z) \: \genjacw^{(a,0)}(z) \: \D A \approx \sum_{j=1}^{M} w_j \: \big[ f(x_j, y_j, z_j) + f(-x_j, -y_j, z_j) \big],
\end{align*}
where $M = M_1 \: M_2$, and the quadrature rule is exact if $f(x,y,z)$ is a polynomial of degree $\le N$ with $M_1 \ge \half (N+1), M_2 \ge N+1$.
\end{theorem}

\begin{remark} 
	Note that the Gauss quadrature nodes and weights $(t_j, w_j^{(t)})$ will have to be calculated, however the Gauss quadrature nodes and weights $(s_j, w_j^{(s)})$ are simply the Chebyshev--Gauss quadrature nodes and weights given explicitly \cite[3.5.23]{DLMF} as $s_j := \cos \fpr(\frac{2j - 1}{2M_2} \pi)$, $w_j^{(s)} := \frac{\pi}{M_2}$.
\end{remark}

\begin{proof}
Let $f : \Omega \to \R$. Define the functions $f_e, f_o : \Omega \to \R$ by 
\begin{align*}
	f_e(x,y,z) &:= \half \Big(f(x, y,z) + f(-x, -y,z)\Big), \quad \forall (x,y,z) \in \Omega\\
	f_o(x,y,z) &:= \half \Big(f(x, y,z) - f(-x, -y,z)\Big), \quad \forall (x,y,z) \in \Omega
\end{align*}
so that $\xvec \mapsto f_e(\xvec, z)$ for fixed $z$ is an even function, and $\xvec \mapsto f_o(\xvec, z)$ for fixed $z$ is an odd function. Note that if $f$ is a polynomial, then $f_e(\rho(t)x, \rho(t)y, t)$ is a polynomial in $t \in [\alpha,1]$ for fixed $(x,y) \in \R^2$. 

Firstly, we note that
\begin{align*}
	\int_0^{2\pi} g\big(\cos \theta, \sin\theta\big) \: \D \theta = \int_{-1}^1 \Big( g\big(x, \sqrt{1-x^2}\big) + g\big(x, -\sqrt{1-x^2}\big) \Big) \frac{\D x}{\sqrt{1-x^2}}
\end{align*}
for some function $g$. Then, integrating the even function $f_e$ we have
\begin{align*}
	&\int_\Omega f_e(x,y,z) \: \genjacw^{(a,0)}(z) \: \D A \\
	&\quad \quad = \int_\alpha^1 \genjacw^{(a,0)}(z) \: \Big( \int_0^{2\pi} f_e\big(\rho(z)\cos\theta,\rho(z)\sin\theta, z\big) \: \D \theta \Big) \: \D z \\
	&\quad \quad = 2 \int_\alpha^1 \genjacw^{(a,0)}(z) \: \Big( \int_{-1}^1 f_e\big(\rho(z)x,\rho(z)\sqrt{1-x^2}, z\big) \: \D x \Big) \: \D z \\
	&\quad \quad \approx \int_\alpha^1 \genjacw^{(a,0)}(z) \: \Big( \sum_{l=1}^{M_2} w_l^{(s)} f_e\big(\rho(z) s_l, \rho(z) \sqrt{1 - s_l^2}, z\big) \Big) \: \D z \quad (\star) \\
	&\quad \quad \approx \sum_{j=1}^{M_1} w_j^{(t)} \sum_{l=1}^{M_2} w_l^{(s)} f_e\big(\rho(t_j) s_l, \rho(t_j) \sqrt{1 - s_l^2}, t_j\big) \quad (\star \star) \\
	&\quad \quad = \sum_{k=1}^{M_1 M_2}  w_j \: f_e(x_j, y_j, z_j).
\end{align*}
Suppose $f$ is a polynomial in $x,y,z$ of degree $N$, and hence that $f_e$ is a degree $\le N$ polynomial. It follows that $s \mapsto f_e\big(\rho(z)s,\rho(z)\sqrt{1-s^2}, z\big)$ for fixed $z$ is then a polynomial of degree $\le N$. We therefore achieve equality at $(\star)$ if $2M_2 - 1 \ge N$ and we achieve equality at $(\star \star)$ if also $2M_1 - 1 \ge N$.

Integrating the odd function $f_o$ results in
\begin{align*}
	&\int_\Omega f_o(x,y,z) \: \genjacw^{(a,0)}(z) \: \D A \\
	&\quad \quad = \int_\alpha^1 \genjacw^{(a,0)}(z) \: \Big( \int_0^{2\pi} f_o\big(\rho(z)\cos\theta,\rho(z)\sin\theta, z)\big) \: \D \theta \Big) \: \D z \\
	&\quad \quad = \int_\alpha^1 \genjacw^{(a,0)}(z) \: \Big( \int_{-1}^1 \Big[ f_o\big(\rho(z)x,\rho(z)\sqrt{1-x^2}, z\big) + f_o\big(\rho(z)x, - \rho(z)\sqrt{1-x^2}, z\big) \big] \: \D x \Big) \: \D z \\
	&\quad \quad = \int_\alpha^1 \genjacw^{(a,0)}(z) \: \Big( \int_{-1}^1 \Big[ f_o\big(\rho(z)x,\rho(z)\sqrt{1-x^2}, z\big) - f_o\big(\rho(z)x, \rho(z)\sqrt{1-x^2}, z\big) \big] \: \D x \Big) \: \D z \\
	&\quad \quad = 0.
\end{align*}
since $f_o(x,y,z) = - f_o(-x, -y, z)$. Hence, for a polynomial $f$ in $x,y,z$ of degree $N$,
\begin{align*}
	\int_\Omega f(x,y,z) \: \genjacw^{(a,0)}(z) \: \D A &= \int_\Omega \Big(f_e(x,y,z) + f_o(x,y,z)\Big) \:  \genjacw^{(a,0)}(z) \: \D A  \\
	&= \int_\Omega f_e(x,y,z) \: \genjacw^{(a,0)}(z) \: \D A \\
	&= \sum_{j=1}^{M}  w_j \: f_e(x_j, y_j, z_j),
\end{align*}
where $M = M_1 M_2$ and $2M_1 - 1 \ge N, 2M_2 - 1 \ge N$.
\end{proof}

\subsection{Obtaining the coefficients for expansion of a function on the spherical cap}\label{subsection:expandingfunctions}

Fix $a \in \R$. Then for any function $f : \Omega \to \R$ we can express $f$ by
\begin{align*}
	f(x,y,z) \approx \sum_{k=0}^N \bigscopNka(x,y,z)^\top \: \bm{f}_k = \bigscopNa(x,y,z)^\top \: \bm{f}
\end{align*}
for N sufficiently large, where $\bigscopNka, \bigscopNa$ is defined in equations (\ref{eqn:OPdefNka}, \ref{eqn:OPdefN0a}, \ref{eqn:OPdefNa}) and where

\begin{align*}
	\bm{f}_k &:= 
		\begin{pmatrix}
			f_{k,k,0} \\
			f_{k,k,1} \\
			\vdots \\
			f_{N,k,0} \\
			f_{N,k,1}
		\end{pmatrix} \in \R^{2(N-k+1)} \quad \text{for } n = 1,2,\dots,N, \quad
	\bm{f}_0 := 
		\begin{pmatrix}
			f_{0,0,0} \\
			\vdots \\
			f_{N,0,0}
		\end{pmatrix} \in \R^{N+1}, \\
	\bm{f} &:= 
		\begin{pmatrix}
			\bm{f}_0 \\
			\vdots \\
			\bm{f}_N
		\end{pmatrix} \in \R^{2(N+1)^2}, \quad \quad
	f_{n,k,i} := \ip< f, \: \scopnkia>_{\scopa} \: \norm{\scopnkia}^{-2}_{\scopa}.
\end{align*}
Recall from equation (\ref{eqn:normscop}) that $\norm{\scopnkia}^2_{\scopa} = \normgenjac^{(a,2k)} \: \pi$. Using the quadrature rule detailed in Section \ref{subsection:quadrule} for the inner product, we can calculate the coefficients $f_{n,k,i}$ for each $n = 0,\dots,N$, $k = 0,\dots,n$, $i = 0,1$: 
\begin{align*}
	f_{n,k,i} &= \frac{1}{2 \: \normgenjac^{(a,2k)} \: \pi} \sum_{j=1}^{M} w_j \big[ f(x_j, y_j, z_j) \scopnkia(x_j, y_j, z_j) +f(-x_j, -y_j, z_j) \scopnkia(-x_j, -y_j, z_j) \big] \\
	&= \frac{1}{M_2 \: \normgenjac^{(a,2k)}} \sum_{j=1}^{M} \big[ f(x_j, y_j, z_j) \scopnkia(x_j, y_j, z_j) +f(-x_j, -y_j, z_j) \scopnkia(-x_j, -y_j, z_j) \big]
\end{align*}
where the quadrature nodes and weights are those from Theorem \ref{Theorem:quadrule}, and $M = M_1 M_2$ with $2M_1 - 1 \ge N, M_2 - 1 \ge N$ (i.e. we can choose $M_2 := N + 1$ and $M_1 := \ceil{\frac{N+1}{2}}$).

\subsection{Function evaluation}\label{subsection:functionevaluation}

For a function $f$, with coefficients vector $\bm{f}$ for expansion in the $\{\scopnki\}$ basis as determined via the method in Section \ref{subsection:expandingfunctions} up to order $N$, we can use the Clenshaw algorithm to evaluate the function at a point $(x,y,z) \in \Omega$ as follows. Let $A_n, B_n, \Dnt, C_n$ be the Clenshaw matrices from Definition \ref{def:clenshawmats}, and define the rearranged coefficients vector $\tilde{\bm{f}}$ via
\begin{align*}
	{\bm{f}}_n &:= 
		\begin{pmatrix}
			f_{n,0,0} \\
			f_{n,1,0} \\
			f_{n,1,1} \\
			\vdots \\
			f_{n,n,0} \\
			f_{n,n,1}
		\end{pmatrix} \in \R^{2(N+1)} \quad \text{for } n = 1,2,\dots,N, \quad
	{\bm{f}}_0 = f_{0,0,0} \in \R, \\
	{\bm{f}} &:= 
		\begin{pmatrix}
			{\bm{f}}_0 \\
			\vdots \\
			{\bm{f}}_N
		\end{pmatrix} \in \R^{(N+1)^2}.
\end{align*}
The trivariate Clenshaw algorithm works similar to the bivariate Clenshaw algorithm introduced in \cite{olver2019triangle} for expansions in the triangle:
\begin{align*}
	\quad &\text{1) } \text{Set } \bm{\xi}_{N+2} = \bold{0}, \: \bm{\xi}_{N+2} = \bold{0}. \\
	\quad &\text{2) } \text{For } n = N:-1:0 \\
	\quad & \quad \quad \quad \text{set } \bm{\xi}_{n}^T = {{\bm{f}}_n}^T - \bm{\xi}_{n+1}^T D^T_n (B_n - G_n(x,y,z)) -  \bm{\xi}_{n+2}^T D^T_{n+1}C_{n+1} \\
	\quad &\text{3) } \text{Output: } f(x,y,z) \approx \bm{\xi}_{0} \: \bigscopta_{0}  = \xi_{0} \: \scopa_0
\end{align*}

\subsection{Calculating non-zero entries of the operator matrices}\label{subsection:Computation-operatormatrices}

The proofs of Theorem \ref{theorem:sparsityofdifferentialoperators} and Lemma \ref{lemma:sparsityofparametertransformationoperators} provide a way to calculate the non-zero entries of the operator matrices given in Definition \ref{def:differentialoperators} and Definition \ref{def:parametertransformationoperators}. We can simply use quadrature to calculate the 1D inner products, which has a complexity of $\bigO(N^2)$. This proves much cheaper computationally than using the 3D quadrature rule to calculate the surface inner products, which has a complexity of $\bigO(N^3)$.

\subsection{Obtaining operator matrices for variable coefficients}\label{subsection:operatorclenshaw}

The Clenshaw algorithm outlined in Section \ref{subsection:functionevaluation} can also be used with Jacobi matrices $J^{(a)}_x, J^{(a)}_y, J^{(a)}_z$ replacing the point $(x,y,z)$. Let $v : \Omega \to \R$ be the function that we wish to obtain an operator matrix $V$ for $v$, so that
\begin{align*}
	v(x,y,z) \: f(x,y,z) = v(x,y,z) \: \bm{f} \: \bigscopta(x,y,z) = (V \bm{f})^\top \: \bigscopta(x,y,z),
\end{align*}
i.e. $V \bm{f}$ is the coefficients vector for the function $v(x,y,z) \: f(x,y,z)$. 

To this end, let $\tilde{\bm{v}}$ be the coefficients for expansion up to order $N$ in the $\{\scopnki\}$ basis of $v$ (rearranged as in Section \ref{subsection:functionevaluation} so that $v(x,y,z) = {\bm{v}}^\top \: \bigscopta(x,y,z)$). Denote $X := (J_x^{(a)})^\top$, $Y := (J_y^{(a)})^\top$, $Z := (J_z^{(a)})^\top$. The operator $V$ is then the result of the following:
\begin{align*}
	\quad &\text{1) } \text{Set } \bm{\xi}_{N+2} = \bold{0}, \: \bm{\xi}_{N+2} = \bold{0}. \\
	\quad &\text{2) } \text{For } n = N:-1:0 \\
	\quad & \quad \quad \quad \text{set } \bm{\xi}_{n}^T = {\tilde{\bm{v}}_n}^T - \bm{\xi}_{n+1}^T D^T_n \big(B_n - G_n(X, Y, Z)\big) -  \bm{\xi}_{n+2}^T D^T_{n+1}C_{n+1} \\
	\quad &\text{3) } \text{Output: } V(X, Y, Z) \approx \bm{\xi}_{0} \: \bigscopta_{0}  = \xi_{0} \: \scopa_0
\end{align*}
where at each iteration, $\bm{\xi}_n$ is a vector of matrices.

\section{Examples on spherical caps with zero Dirichlet conditions}\label{Section:Examples}

We now demonstrate how the sparse linear systems constructed as above can be used to efficiently solve PDEs with zero Dirichlet conditions on the spherical cap defined by $\Omega$. We consider Poisson, inhomogeneous variable coefficient Helmholtz equation and the Biharmonic equation, as well as a time dependent heat equation, demonstrating the versatility of the approach.

\subsection{Poisson}


\begin{figure}[t]
	\begin{subfigure}{0.5\textwidth}
	\includegraphics[scale=0.25]{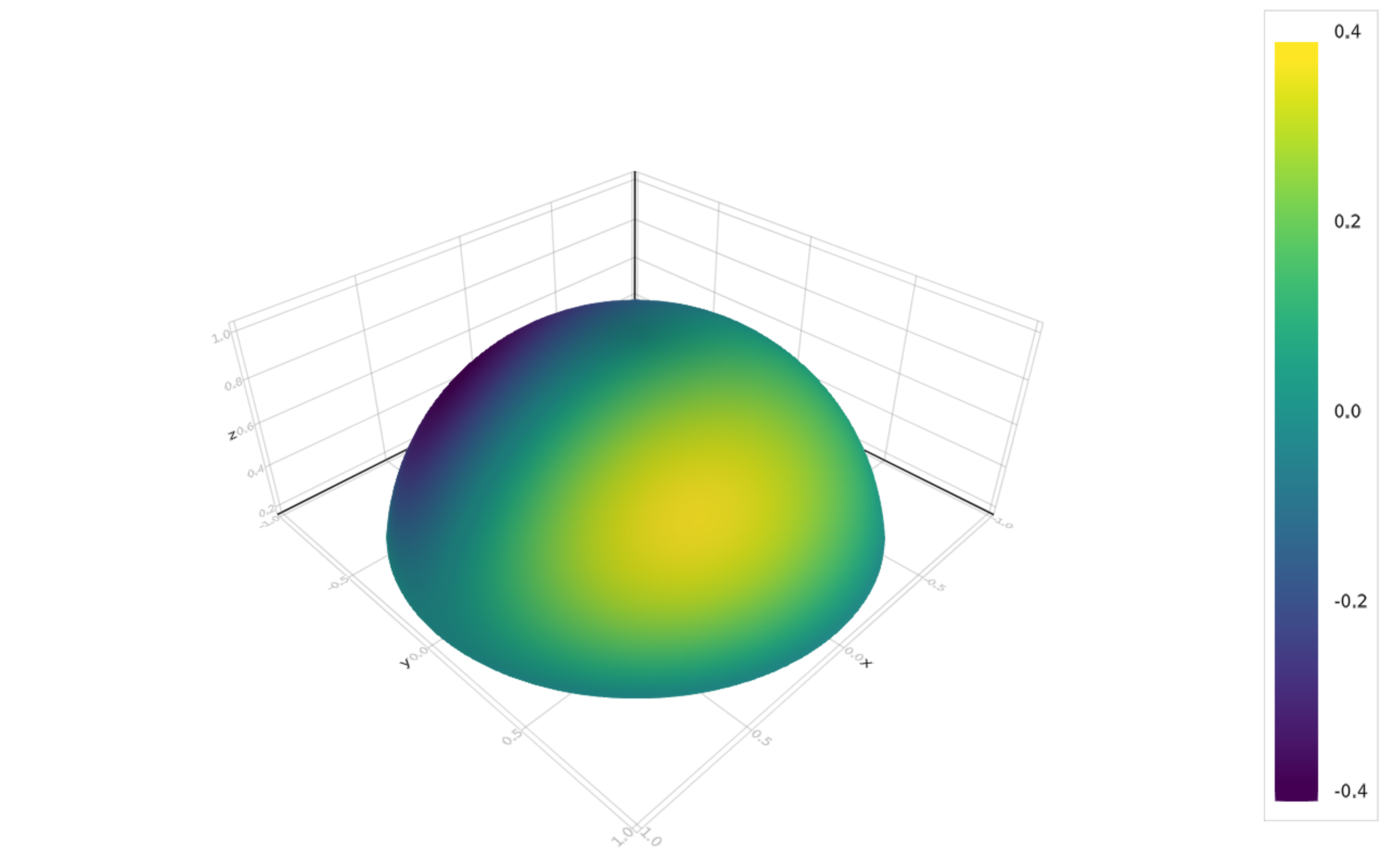}
	\centering
	\end{subfigure}
	\begin{subfigure}{0.5\textwidth}
	\centering
	\includegraphics[scale=0.3]{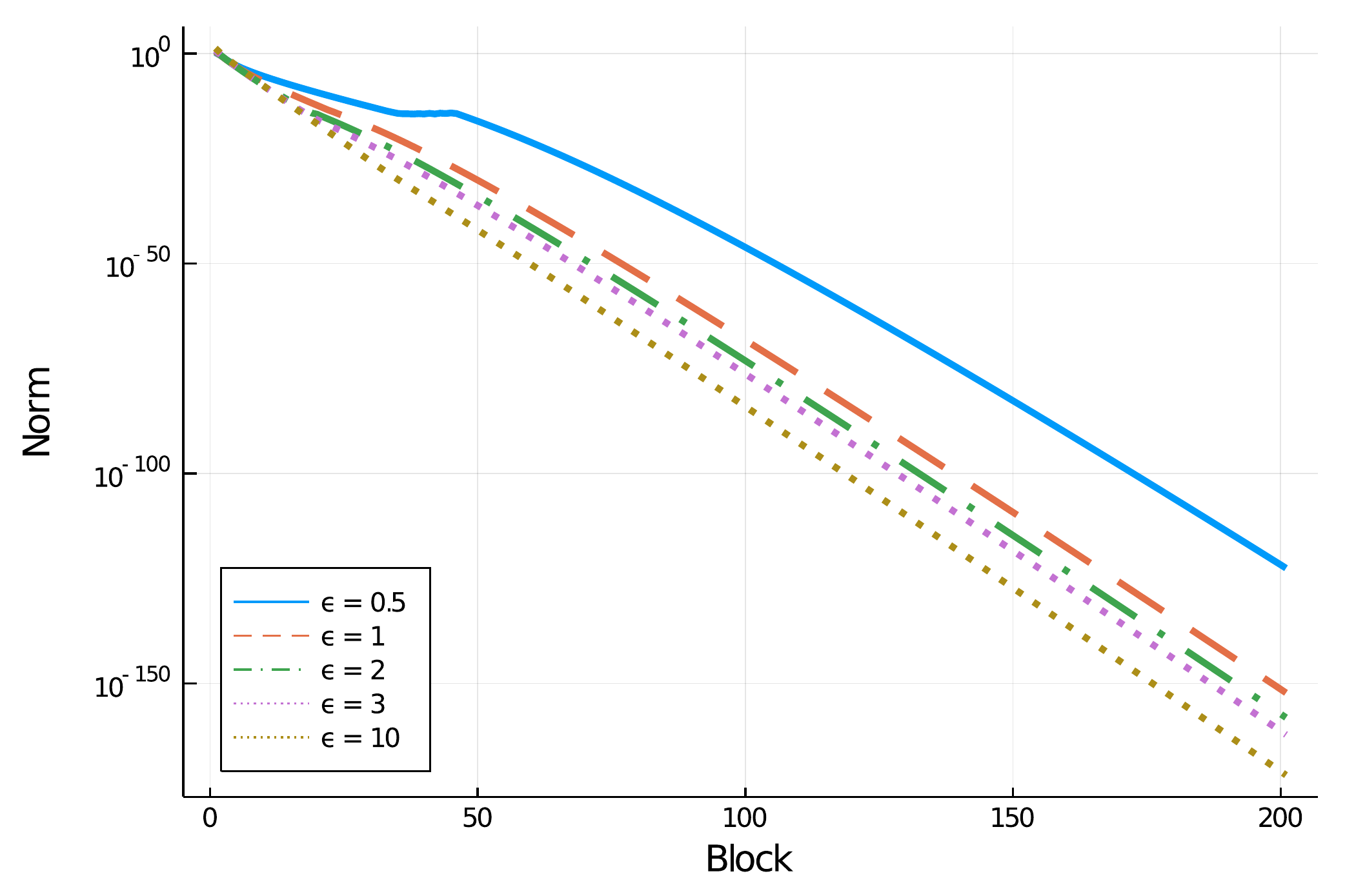}
	\end{subfigure}
	\caption{Left: The computed solution to $\Delta u = f$ with zero boundary conditions with $f(x,y,z) = - 2 e^x y z (2+x) + \genjacw^{(1,0)}(z) e^x (y^3 + z^2 y - 4xy - 2y)$. Right: The norms of each block of the computed solution of the Poisson equation with right hand side function $f(x,y,z) = \norm{\xvec - (\epsilon + 1 / \sqrt{3}) \; (1,1,1)^\top }$ for different $\epsilon$ values. This indicates spectral convergence.}
	\centering
	\label{fig:poisson}
\end{figure}

\begin{figure}[t]
	\begin{subfigure}{0.5\textwidth}
	\includegraphics[scale=0.25]{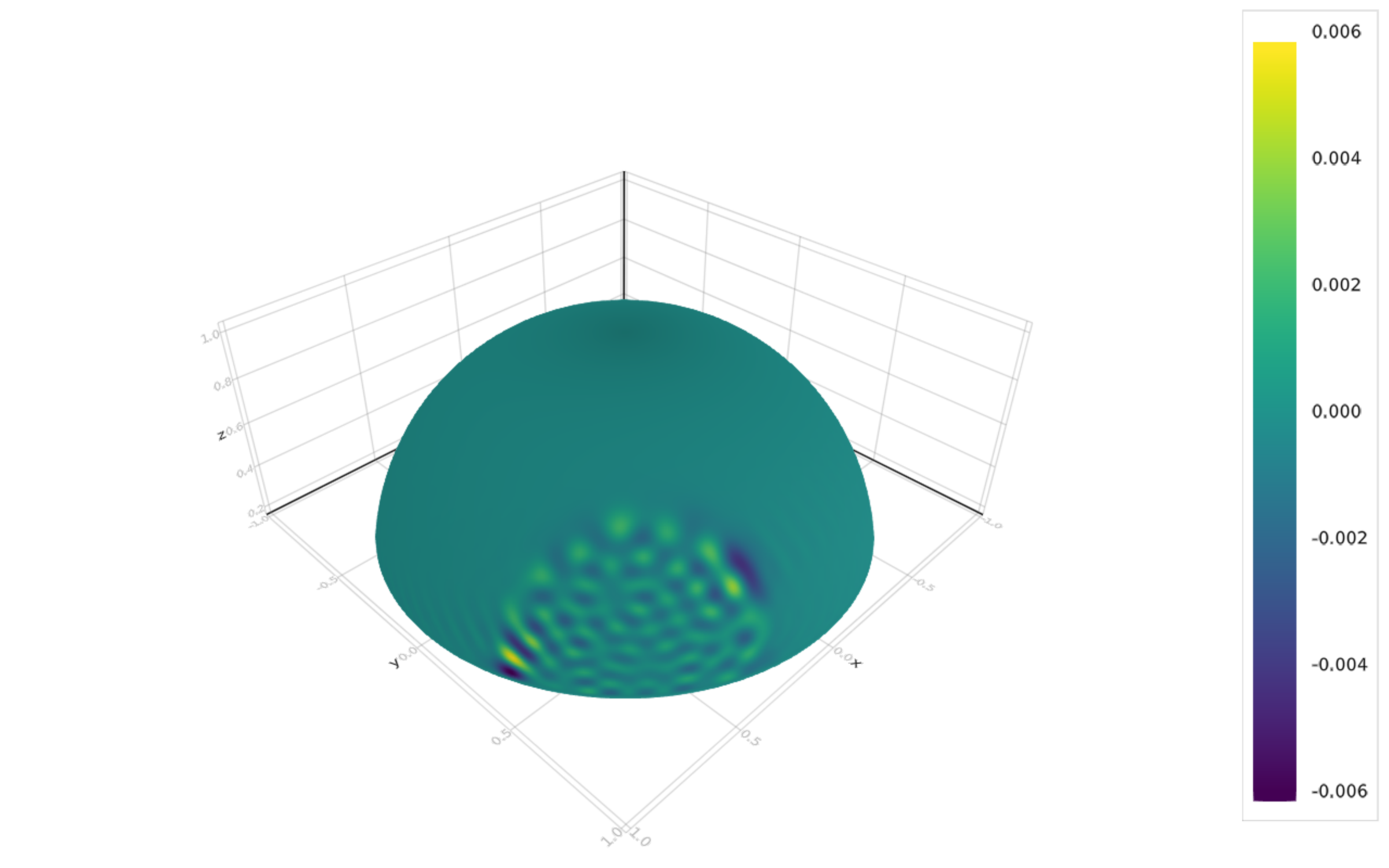}
	\centering
	\end{subfigure}
	\begin{subfigure}{0.5\textwidth}
	\centering
	\includegraphics[scale=0.3]{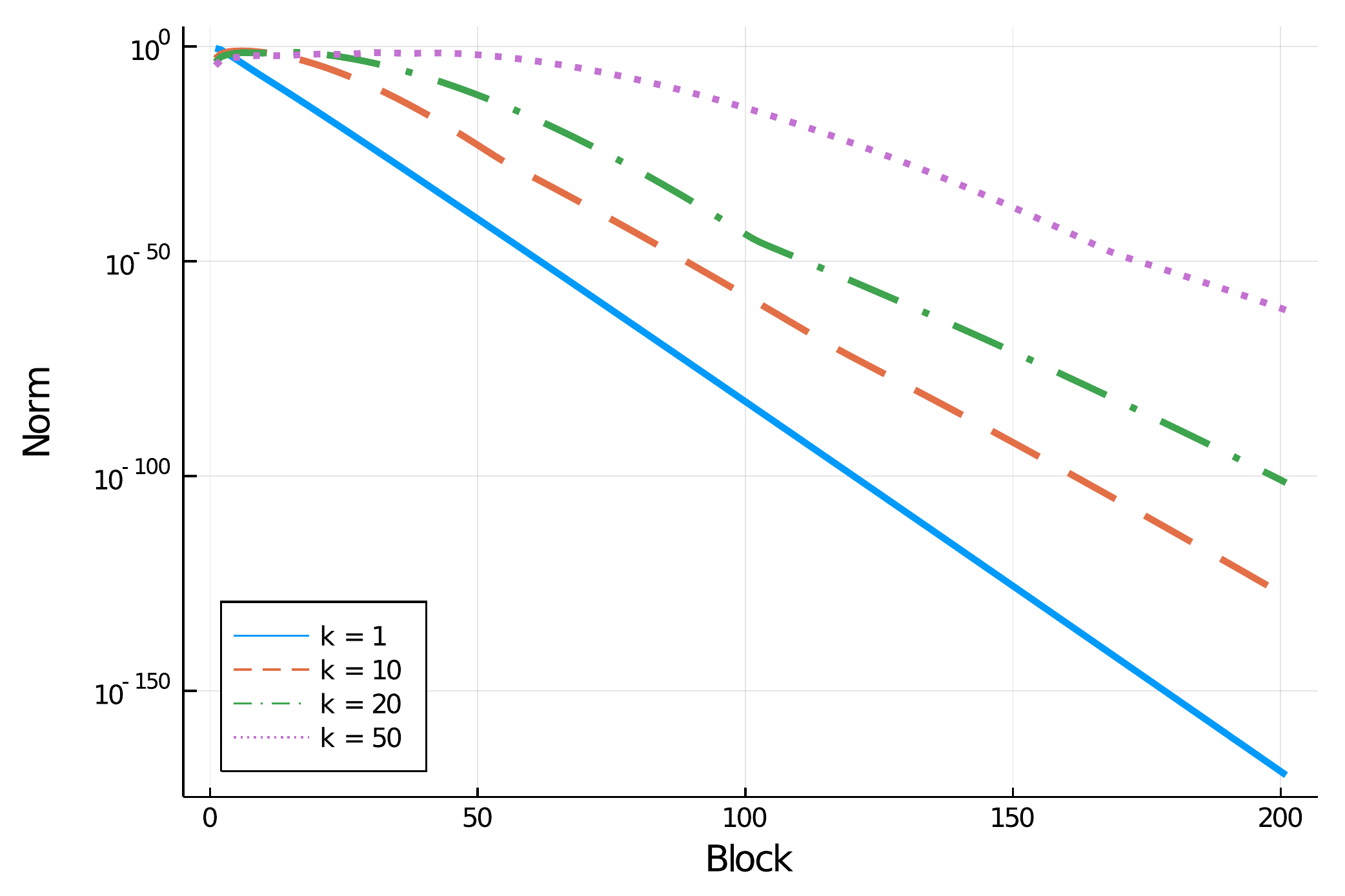}
	\end{subfigure}
	\caption{Left: The computed solution to $\Delta u + k^2 \: v \: u = f$ with zero boundary conditions with $f(x,y,z) = y e^x (z - \alpha) $, $v(x,y,z) = 1 - (3(x-x_0)^2 + 5(y-y_0)^2 + 2(z-z_0)^2)$ where $(x_0, z_0) := (0.7, 0.2)$, $y_0 := \sqrt{1 - x_0^2 - z_0^2}$ and $k = 100$. Right: The norms of each block of the computed solution of the Helmholtz equation with the right hand side function $f(x,y,z) = 1$ and the same function $v(x,y,z)$, for various $k$ values. This indicates spectral convergence.}
	\centering
	\label{fig:helmholtz}
\end{figure}

\begin{figure}[t]
	\centering 
	\includegraphics[scale=0.35]{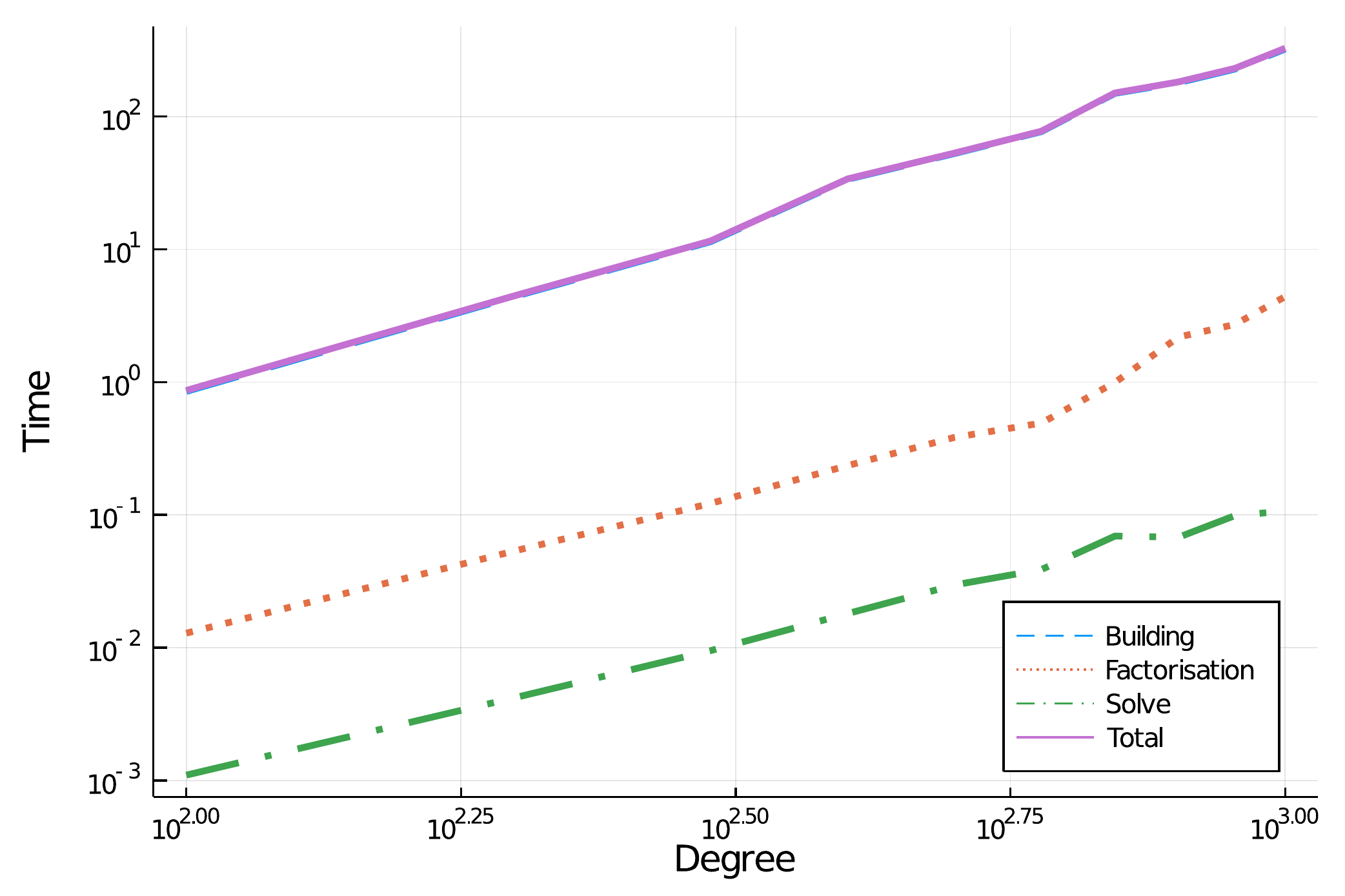}
	\caption{Time in seconds to build and solve the system $\big[\DeltaS + v(x,y,z) \big] \: u(x,y,z) = f(x,y,z)$, for a rotationally invariant $v(x,y,z) = v(z)$. This demonstrates that the approach is roughly of order $\bigO(N^2)$, where $N$ is the degree to which we approximate the solution. Here, we used $f = -2e^{x}yz(2+x) + (z - \alpha) e^{x} (y^3 + z^2 y - 4xy - 2y)$ and $v(x,y,z) = v(z) = \cos(z)$.}
	\centering
	\label{fig:complexity}
\end{figure}

\begin{figure}[t]
	\begin{subfigure}{0.5\textwidth}
	\includegraphics[scale=0.15]{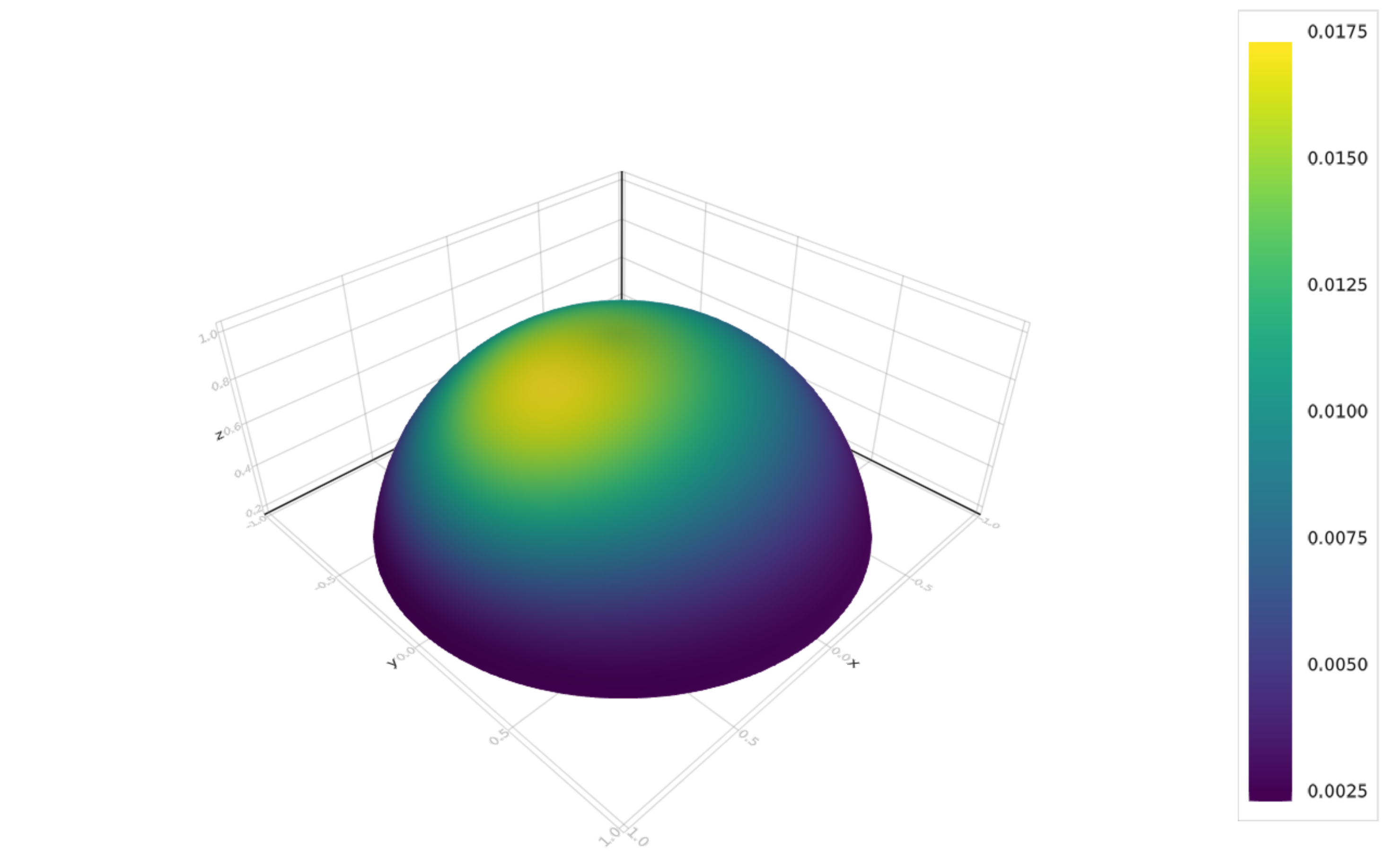}
	\centering
	\end{subfigure}
	\begin{subfigure}{0.5\textwidth}
	\includegraphics[scale=0.3]{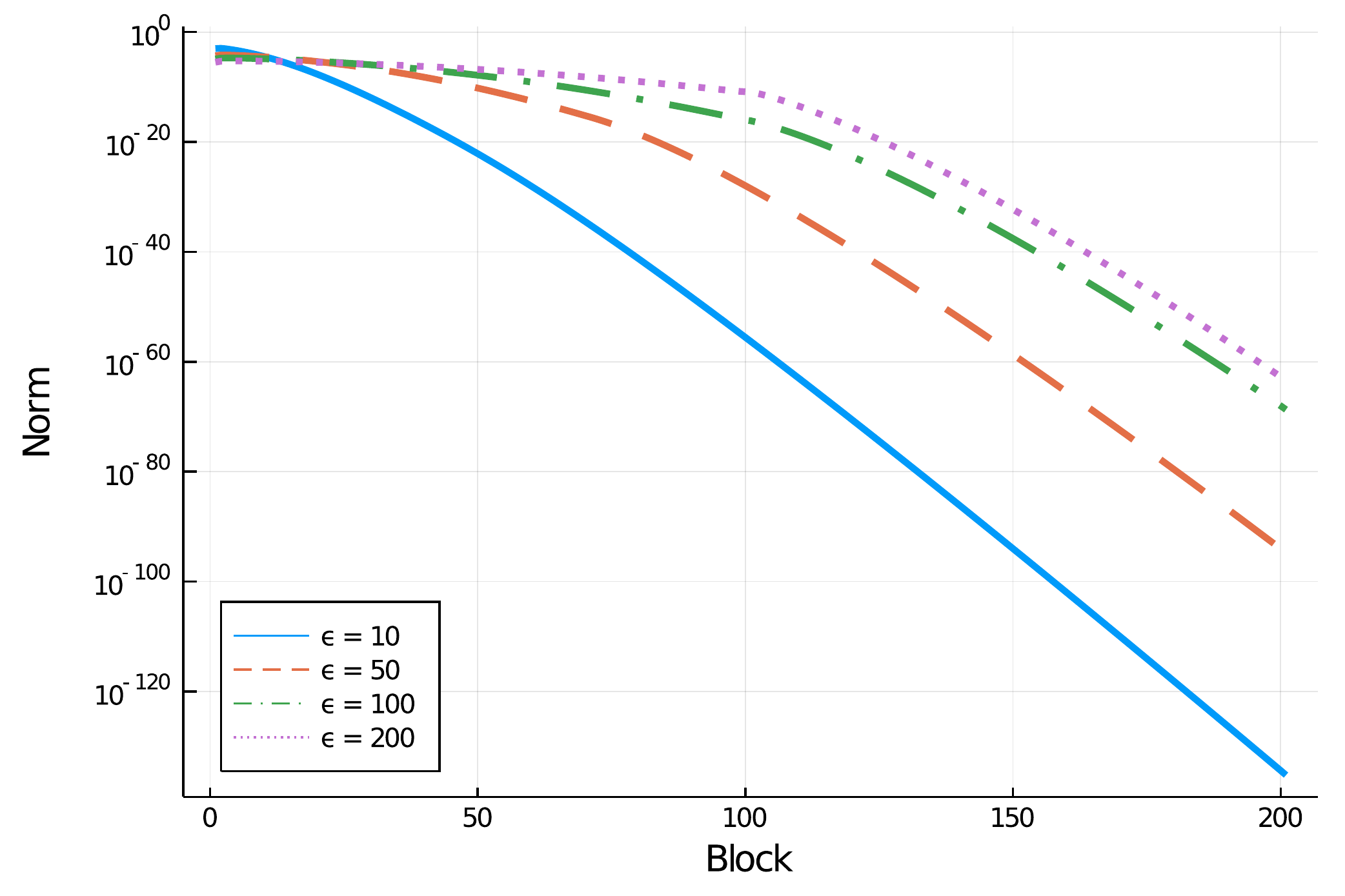}
	\centering
	\end{subfigure}
	\caption{Left: The computed solution to $\Delta^2 u = f$ with zero Dirichlet and Neumann boundary conditions with $f(x,y,z) = (1 + \text{erf}(5(1 - 10((x - 0.5)^2 + y^2)))) \rho(z)^2$. Right: The norms of each block of the computed solution of the biharmonic equation with the right hand side function $f(x,y,z) = \exp(-\epsilon((x-x_0)^2 + (y-y_0)^2 + (z-z_0)^2))$ where  $(x_0, z_0) := (0.7, 0.2)$, $y_0 := \sqrt{1 - x_0^2 - z_0^2}$, for various $\epsilon$ values. This demonstrates algebraic convergence.}
	\centering
	\label{fig:biharmonic}
\end{figure}

The Poisson equation is the classic problem of finding $u(x,y,z)$ given a function $f(x,y,z)$ such that:
\begin{align}
	\begin{cases}
\DeltaS u(x,y,z) = f(x,y,z) &\quad \text{in } \Omega \\
		u(x,y,z) = 0& \quad \text{on } \partial \Omega
	\end{cases}
	\label{eqn:poisson}
\end{align}
noting the imposition of zero Dirichlet boundary conditions on $u$.

We can tackle the problem as follows. Choose an $N \in \N$ large enough for the problem, and denote the coefficient vector for expansion of $u$ in the $\bigWNi$ OP basis up to degree $N$ by $\mathbf{u}$, and the coefficient vector for expansion of $f$ in the $\bigscopNi$ OP basis up to degree $N$ by $\mathbf{f}$. Since $f$ is known, we can obtain $\mathbf{f}$ using the quadrature rule in Section \ref{subsection:expandingfunctions}. In matrix-vector notation, our system hence becomes:
\begin{align*}
    \Delta_W^{(1)} \mathbf{u} = \mathbf{f}
\end{align*}
which can be solved to find $\mathbf{u}$.
In Figure \ref{fig:poisson} we see the solution to the Poisson equation with zero boundary conditions given in (\ref{eqn:poisson}) in the disk-slice $\Omega$. In Figure \ref{fig:poisson} we also show the norms of each block of calculated coefficients of the approximation for four right-hand sides of the Poisson equation with $N = 200$, that is, $(N+1)^2 = 40,401$ unknowns. The right hand sides we choose here are given by
\begin{align*}
	f(x,y,z) &= \norm{\Big(x - (\epsilon + 1 / \sqrt{3}), \: y - (\epsilon + 1 / \sqrt{3}), \: z - (\epsilon + 1 / \sqrt{3})\Big)^\top}
\end{align*}
for differing choices of $\epsilon$ -- this parameter serves to alter the distance from which we would have a singularity. In the plot, a \enquote{block} is simply the group of coefficients corresponding to OPs of the same degree, and so the plot shows how the norms of these blocks decay as the degree of the expansion increases. Thus, the rate of decay in the coefficients is a proxy for the rate of convergence of the computed solution: as typical of spectral methods, we expect the numerical scheme to converge at the same rate as the coefficients decay. We see that we achieve spectral convergence for these examples.

\subsection{Inhomogeneous variable-coefficient Helmholtz}

Find $u(x,y)$ given functions $v$, $f : \Omega \to \R$ such that:
\begin{align}
	\begin{cases}
    		\DeltaS u(x,y,z) + k^2 \: v(x,y,z) \; u(x,y,z) = f(x,y,z) &\quad \text{in } \Omega \\
		u(x,y,z) = 0 &\quad \text{on } \partial \Omega
	\end{cases}
	\label{eqn:helmholtz}
\end{align}
where $k \in \R$, noting the imposition of zero Dirichlet boundary conditions on $u$.

We can tackle the problem as follows. Denote the coefficient vector for expansion of $u$ in the $\bigWNi$ OP basis up to degree $N$ by $\mathbf{u}$, and the coefficient vector for expansion of $f$ in the $\bigscopNi$ OP basis up to degree $N$ by $\mathbf{f}$. Since $f$ is known, we can obtain  the coefficients $\mathbf{f}$ using the quadrature rule in Section \ref{subsection:expandingfunctions}. 

Define $X := (J_x^{(0)})^\top$, $Y := (J_y^{(0)})^\top$, $Z := (J_z^{(0)})^\top$. We can obtain the matrix operator for the variable-coefficient function $v(x,y,z)$ by using the Clenshaw algorithm with matrix inputs as the Jacobi matrices $X, Y, Z$, yielding an operator matrix of the same dimension as the input Jacobi matrices a la the procedure introduced in \cite{olver2019triangle}. We can denote the resulting operator acting on coefficients in the $\bigscopNo$ space by $v(X, Y, Z)$. In matrix-vector notation, our system hence becomes:
\begin{align*}
    (\Delta_W^{(1)} + k^2 \:T^{(0)\to(1)} \: V \: T_W^{(1)\to(0)}) \: \mathbf{u} = \mathbf{f}
\end{align*}
which can be solved to find $\mathbf{u}$. We can see the sparsity and structure of this matrix system in Figure \ref{fig:sparsity} with $v(x,y,z) = zxy^2$ as an example. In Figure \ref{fig:helmholtz} we see the solution to the inhomogeneous variable-coefficient Helmholtz equation with zero boundary conditions given in (\ref{eqn:helmholtz}) in the spherical cap $\Omega$, with $f(x,y,z) = y e^x \genjacw^{(1,0)}(z)$, $v(x,y,z) = 1 - (3(x-x_0)^2 + 5(y-y_0)^2 + 2(z-z_0)^2)$ where $(x_0, z_0) := (0.7, 0.2)$, $y_0 := \sqrt{1 - x_0^2 - z_0^2}$ and $k = 100$. In Figure \ref{fig:helmholtz} we also show the norms of each block of calculated coefficients for the approximation of the solution to the inhomogeneous variable-coefficient Helmholtz equation with various $k$ values. Here, we use $N = 200$, that is, $(N+1)^2 = 40,401$ unknowns. Once again, the rate of decay in the coefficients is a proxy for the rate of convergence of the computed solution, and we see that we achieve spectral convergence.

In Figure~\ref{fig:complexity} we plot the time taken\footnote{measured using the \enquote{@belapsed} macro from the BenchmarkTools.jl package \cite{BenchmarkTools.jl-2016} in Julia.} to construct the operator for $\DeltaS + v(x,y,z)$, with a rotationally invariant $v(x,y,z) = v(z) = \cos z$, and solve a zero boundary condition Helmholtz problem. The plot demonstrates that as we increase the degree of approximation $N$, we achieve a complexity of an optimal $\bigO(N^2)$.

What about other boundary conditions? One simple extension is the case where the value on the boundary takes that of a function depending only on $x$ and $y$, i.e. $c = c(x,y)$. In this case, the problem
\begin{align*}
	\begin{cases}
    		\DeltaS u(x,y,z) + k^2 \: v(x,y,z) \; u(x,y,z) = f(x,y,z) &\quad \text{in } \Omega \\
		u(x,y,z) = c(x,y) &\quad \text{on } \partial \Omega
	\end{cases}
\end{align*}
is equivalent to letting $u(x,y,z) = \tilde{u}(x,y,z)+ c(x,y)$ and solving
\begin{align*}
	\begin{cases}
    		\DeltaS \tilde{u}(x,y,z) + k^2 \: v(x,y,z) \; \tilde{u}(x,y,z) = f(x,y,z) - k^2 \: v(x,y,z) \: c(x,y) - \DeltaS \: c(x,y)  \quad \text{in } \Omega \\
		\tilde{u}(x,y,z) = 0 \quad \text{on } \partial \Omega
	\end{cases}
\end{align*}
for $\tilde u$. This new problem is then a zero boundary condition Helmholtz problem with right hand side 
\begin{align*}
	g(x,y,z) &:= f(x,y,z) - k^2 \: v(x,y,z) \: c(x,y) - \DeltaS \: c(x,y)
\end{align*}
for $(x,y,z) \in \Omega$. Notice that the spherical Laplacian applied to $c(x,y)$, expanded in the $\bigscopN^{(1)}$ basis with coefficients vector $\vec c = (c_{n,k,i})$, is just
\begin{align*}
	\DeltaS \: c(x,y) = {1 \over \rho(z)^2} \sum_{n=0}^N \sum_{i=0}^1 c_{n,n,i} \: \pmpxm{\theta}{2} \ch_{n,i}(\theta) = - {1 \over \rho(z)^2} \sum_{n=0}^N \sum_{i=0}^1 n^2 \: c_{n,n,i} \: \ch_{n,i}(\theta)
\end{align*}
since the coefficients $\{c_{n,k,i}\}$ for such a function are zero for $k < n$ due to the dependence on $x$ and $y$ only, which are precisely the Fourier coefficients of $c(\cos \theta, \sin \theta)$. Thus, since the function $c(x,y)$ is known, it is simple to evaluate $\pmpxm{\theta}{2} c(x,y)$ and hence one can obtain the coefficients for the expansion of $g(x,y,z)$ in the $\bigscopN^{(1)}$ basis in the usual manor.

\subsection{Biharmonic equation}

Our last erxample is the biharmonic equation: find $u(x,y,z)$ given a function $f(x,y,z)$ such that:
\begin{align}
	\begin{cases}
    		\DeltaS^2 u(x,y,z) = f(x,y,z) &\quad \text{in } \Omega \\
		u(x,y,z) = 0, \quad \frac{\partial u}{\partial n}(x,y,z) = \grad_S \: u(x,y,z) \cdot \unitvec{n}(x,y,z) = 0 &\quad \text{on } \partial \Omega
	\end{cases}
	\label{eqn:biharmonic}
\end{align}
where $\DeltaS^2$ is the Biharmonic operator, noting the imposition of zero Dirichlet and Neumann boundary conditions on $u$. For clarity, we reiterate that the unit normal vector in this sense is simply $\unitvec{n}(x,y,z) = \unitvec{n}(\vec x) := {\vec x \over \norm{\vec x}} = \vec x$ (see Section \ref{subsection:furtherdiffoperators}). In Figure \ref{fig:biharmonic} we see the solution to the Biharmonic equation (\ref{eqn:biharmonic}) in the spherical cap $\Omega$. In Figure \ref{fig:biharmonic} we also show the norms of each block of calculated coefficients of the approximation for four more complex right-hand sides of the biharmonic equation with $N = 200$, that is, $(N+1)^2 = 40,401$ unknowns. Once again, the rate of decay in the coefficients is a proxy for the rate of convergence of the computed solution, and we see that we achieve exponential convergence for these more complex functions.

\section{Conclusions}

We have shown that trivariate orthogonal polynomials can lead to sparse discretizations of general linear PDEs on spherical cap domains, with Dirichlet boundary conditions on the $z = \alpha \in (0,1)$ boundary. We have provided a detailed practical framework for the application of the methods described for quadratic surfaces of revolution \cite{olver2020orthogonal}, by utilising the non-classical 1D OPs on the interval $[\alpha, 1]$ with the weight $(z - \alpha)^a \: (1-z^2)^{b/2}$ defined for the disk-slice case \cite{snowball2019sparse}. Generalisation to spherical bands ($\alpha \leq z \leq \beta$) is straightforward.  This work thus forms a building block in developing an $hp-$finite element method to solve PDEs on the sphere by using spherical band and spherical cap shaped elements.

This work also serves as a stepping stone to constructing similar methods to solve partial differential equations on other 3D sub-domains of the sphere---it is clear from the construction in this paper that discretizations of spherical gradients and Laplacian's are sparse on other suitable sub-components of the sphere. The resulting sparsity in high-polynomial degree discretizations presents an attractive alternative to methods based on bijective mappings (e.g., \cite{DGShallowWater,FEMShallowWater,boyd2005sphere}). Constructing these sparse spectral methods for surface PDEs on  spherical triangles is future work, and has applications in weather prediction \cite{staniforth2012horizontal}, though it is not yet clear how to directly construct the necessary orthogonal polynomials. 

The next stage is to develop an orthogonal basis for the tangent space of the spherical cap (or band), and obtain sparse differential operators for gradient, divergence etc. On the complete sphere, the vector spherical harmonics that form the orthogonal basis are simply the gradients and perpendicular gradients of the scalar spherical harmonics \cite{barrera1985vector} which has been used effectively for solving PDEs on the sphere \cite{vasil2019tensor,lecoanet2019tensor} -- however, we do not have that luxury for the spherical cap or band, and hence the choice of basis will not be as straightforward.

\bibliography{spherical-caps}

\begin{thebibliography}{10}

\bibitem{barrera1985vector}
Rub{\'e}n~G Barrera, GA~Estevez, and J~Giraldo.
\newblock {V}ector spherical harmonics and their application to magnetostatics.
\newblock {\em European Journal of Physics}, 6(4):287, 1985.

\bibitem{DGShallowWater}
Boris Bonev, Jan~S Hesthaven, Francis~X Giraldo, and Michal~A Kopera.
\newblock Discontinuous {G}alerkin scheme for the spherical shallow water
  equations with applications to tsunami modeling and prediction.
\newblock {\em Journal of Computational Physics}, 362:425--448, 2018.

\bibitem{boyd2005sphere}
John~P Boyd.
\newblock A {C}hebyshev/rational {C}hebyshev spectral method for the
  {H}elmholtz equation in a sector on the surface of a sphere: defeating corner
  singularities.
\newblock {\em Journal of Computational Physics}, 206(1):302--310, 2005.

\bibitem{BenchmarkTools.jl-2016}
Jiahao {Chen} and Jarrett {Revels}.
\newblock Robust benchmarking in noisy environments.
\newblock {\em arXiv e-prints}, Aug 2016.

\bibitem{dunkl2014orthogonal}
Charles~F Dunkl and Yuan Xu.
\newblock {\em {O}rthogonal {P}olynomials of {S}everal {V}ariables}.
\newblock Number 155. Cambridge University Press, 2014.

\bibitem{lecoanet2019tensor}
Daniel Lecoanet, Geoffrey~M Vasil, Keaton~J Burns, Benjamin~P Brown, and
  Jeffrey~S Oishi.
\newblock Tensor calculus in spherical coordinates using jacobi polynomials.
  {P}art-{II}: Implementation and examples.
\newblock {\em Journal of Computational Physics: X}, 3:100012, 2019.

\bibitem{magnus1995painleve}
Alphonse~P Magnus.
\newblock {P}ainlev{\'e}-type differential equations for the recurrence
  coefficients of semi-classical orthogonal polynomials.
\newblock {\em Journal of Computational and Applied Mathematics},
  57(1-2):215--237, 1995.

\bibitem{DLMF}
Frank~WJ Olver, Daniel~W Lozier, Ronald~F Boisvert, and Charles~W Clark.
\newblock {\em {NIST} {H}andbook of {M}athematical {F}unctions}.
\newblock Cambridge University Press, 2010.

\bibitem{olver2013fast}
Sheehan Olver and Alex Townsend.
\newblock A fast and well-conditioned spectral method.
\newblock {\em SIAM Review}, 55(3):462--489, 2013.

\bibitem{olver2019triangle}
Sheehan Olver, Alex Townsend, and Geoff Vasil.
\newblock A sparse spectral method on triangles.
\newblock {\em SIAM J. Sci. Comput.}, 41(6):A3728--A3756, 2019.

\bibitem{olver2018recurrence}
Sheehan Olver, Alex Townsend, and Geoffrey~M Vasil.
\newblock Recurrence relations for a family of orthogonal polynomials on a
  triangle.
\newblock In {\em Spectral and High Order Methods for Partial Differential
  Equations ICOSAHOM 2018}, pages 79--92. Springer, Cham, 2020.

\bibitem{olver2020orthogonal}
Sheehan Olver and Yuan Xu.
\newblock Orthogonal polynomials in and on a quadratic surface of revolution.
\newblock {\em Mathematics of Computation}, 89:2847--2865, 2020.

\bibitem{FEMShallowWater}
J~Shipton, TH~Gibson, and CJ~Cotter.
\newblock Higher-order compatible finite element schemes for the nonlinear
  rotating shallow water equations on the sphere.
\newblock {\em Journal of Computational Physics}, 375:1121--1137, 2018.

\bibitem{snowball2019sparse}
Ben Snowball and Sheehan Olver.
\newblock Sparse spectral and p-finite element methods for partial differential
  equations on disk slices and trapeziums.
\newblock {\em Studies in Applied Mathematics}, 145:3--35, 2020.

\bibitem{staniforth2012horizontal}
Andrew Staniforth and John Thuburn.
\newblock Horizontal grids for global weather and climate prediction models: a
  review.
\newblock {\em Quarterly Journal of the Royal Meteorological Society},
  138(662):1--26, 2012.

\bibitem{vasil2016tensor}
Geoffrey~M Vasil, Keaton~J Burns, Daniel Lecoanet, Sheehan Olver, Benjamin~P
  Brown, and Jeffrey~S Oishi.
\newblock Tensor calculus in polar coordinates using {J}acobi polynomials.
\newblock {\em Journal of Computational Physics}, 325:53--73, 2016.

\bibitem{vasil2019tensor}
Geoffrey~M Vasil, Daniel Lecoanet, Keaton~J Burns, Jeffrey~S Oishi, and
  Benjamin~P Brown.
\newblock Tensor calculus in spherical coordinates using {J}acobi polynomials.
  {P}art-{I}: Mathematical analysis and derivations.
\newblock {\em Journal of Computational Physics: X}, 3:100013, 2019.

\end{thebibliography}

\end{document}